\documentclass[11pt,letterpaper]{article}
\usepackage{graphicx} 
\usepackage[margin=1in]{geometry}

\usepackage{amsmath}
\usepackage{amssymb}
\usepackage{amsthm}
\usepackage{hyperref}
\usepackage{xcolor}
\usepackage{epsfig}
\usepackage{enumitem}

\usepackage{enumitem}
\setlist{leftmargin=6mm,nolistsep,noitemsep}

\usepackage[capitalize,noabbrev]{cleveref}
\crefname{question}{Question}{Questions}
\crefname{step}{Step}{Steps}
\crefname{claim}{Claim}{Claims}
\crefname{problem}{Problem}{Problems}
\crefname{definition}{Definition}{Definitions}
\crefname{observation}{Observation}{Observations}

\DeclareMathOperator{\conv}{conv}

\newcommand{\MP}{\text{MP}}
\newcommand{\QP}{\text{QP}}
\newcommand{\BQP}{\text{BQP}}
\newcommand{\PP}{\text{PP}}

\def\tw{{\rm tw}}
\def\poly{{\rm poly}}

\newcommand{\ie}{i.e., }
\renewcommand{\S}{\mathcal S}

\def\R{{\mathbb R}}

\def\Q{{\mathcal Q}}
\def\X{{\mathcal X}}
\def\A{{\mathcal A}}

\newtheorem{theorem}{Theorem}
\newtheorem{corollary}{Corollary}
\newtheorem{proposition}{Proposition}
\newtheorem{example}{Example}
\newtheorem{observation}{Observation}

\newtheorem{lemma}{Lemma}

\def\P{{\mathcal P}}

\newcommand{\C}{\mathcal C}

\newcommand{\NP}{\mathcal {NP}}
\newcommand{\D}{\mathcal D}

\usepackage[title]{appendix}

\title{A second-order cone representable class of nonconvex quadratic programs}

\author{Santanu S. Dey
\thanks{ H. Milton Stewart School of Industrial and Systems Engineering,
   Georgia Institute of Technology.
             E-mail: {\tt santanu.dey@isye.gatech.edu}.
             }
\and
Aida Khajavirad
\thanks{Department of Industrial and Systems Engineering,
             Lehigh University.
             E-mail: {\tt aida@lehigh.edu}.
             }
}

\begin{document}

\maketitle

\begin{abstract}
    We consider the problem of minimizing a sparse nonconvex quadratic function over the unit hypercube. By developing an extension of the Reformulation-Linearization Technique (RLT) to continuous quadratic sets, we propose a novel second-order cone (SOC) representable relaxation for this problem. By exploiting the sparsity of the quadratic function, we establish a sufficient condition under which the convex hull of the feasible region of the lifted quadratic program is SOC-representable.
    While the proposed formulation may be of exponential size in general, we identify additional structural conditions that guarantee the existence of a polynomial-size SOC-representable formulation, which can be constructed in polynomial time. Under these conditions, the optimal value of the nonconvex quadratic program coincides with that of a polynomial-size second-order cone program.
    Our results serve as a starting point for bridging the gap between the Boolean quadric polytope of sparse problems and its continuous counterpart.
\end{abstract}

\emph{Key words:} nonconvex quadratic programming, convex hull, second-order cone representable, Boolean quadric polytope, polynomial-size extended formulation.

\section{Introduction}

We consider a nonconvex box-constrained quadratic program:
\begin{align}\label[problem]{pQP}
\tag{QP}
\min \quad & x^\top Q x + c^\top x \\
{\rm s.t.} \quad & x \in [0,1]^n, \nonumber
\end{align}
where $c \in \R^n$ and $Q \in \R^{n\times n}$ is a symmetric matrix. It is well known that~\cref{pQP} is $\NP$-hard in general~\cite{HorTuy96}. If $Q$ is positive-semidefinite, then~\cref{pQP} is a convex optimization problem and can be solved in polynomial-time. Henceforth, we assume that $Q$ is not positive semidefinite.
Following a common practice in nonconvex optimization, we linearize the objective function of~\cref{pQP} by introducing new variables $Y := x x^\top$, thus obtaining a reformulation of this problem in a lifted space of variables: 
\begin{align}\label[problem]{lQP}
\tag{$\ell$QP}
\min \quad & \langle Q,Y\rangle  + c^\top x \\
{\rm s.t.} \quad & Y = x x^\top \nonumber\\ 
& x \in [0,1]^n, \nonumber
\end{align}
where $\langle Q,Y\rangle$ denotes the matrix inner product.
Throughout this paper, given a set $\C$, we denote by $\conv(\C)$, the convex hull of the set $\C$.
We are interested in obtaining sufficient conditions under which~\cref{lQP} can be solved via a polynomial-size convex relaxation. We define
$$
\QP_n := \conv\Big\{(x, Y) \in \R^{n+\frac{n(n+1)}{2}}: Y = x x^\top, \; x \in [0,1]^n\Big\}.
$$
In~\cite{BurLet09}, the authors study some fundamental properties of $\QP_n$ and investigate the strength of existing relaxations for this set.

\subsection{Semidefinite programming relaxations}
Semidefinite programming (SDP) relaxations are perhaps the most popular convex relaxations for nonconvex quadratic programs and quadratically-constrained quadratic programs~\cite{Ye10,BaoSahTaw11}. The core idea is to replace the nonconvex constraint $Y=xx^\top$ by the convex relaxation $Y\succeq xx^\top$, to obtain the following basic SDP relaxation of~\cref{lQP}: 
\begin{align}\label[problem]{bSDP}
\tag{bSDP}
\min \quad & \langle Q,Y\rangle  + c^\top x \\
{\rm s.t.} \quad & \begin{bmatrix}
1  & x^\top \\
x & Y
\end{bmatrix} \succeq 0 \nonumber\\ 
& {\rm diag}(Y) \leq x \nonumber\\
& x \in [0,1]^n, \nonumber
\end{align}
where ${\rm diag}(Y)$ denotes the vector in $\R^n$ containing the diagonal entries of $Y$. 
It then follows that a convex relaxation of $\QP_n$ is given by:
$$
\C_n^{\rm SDP} := \Big\{(x, Y) \in \R^{n+\frac{n(n+1)}{2}}:\begin{bmatrix}
1  & x^\top \\
x & Y
\end{bmatrix} \succeq 0, \; {\rm diag}(Y) \leq x, \; x \in [0,1]^n\Big\}.
$$
In~\cite{KimKoj03}, the authors proved that 
if the off-diagonal entries of $Q$ are nonpositive and the vector $c$ is entry-wise nonpositive, then the optimal value of~\cref{pQP} equals the optimal value of~\cref{bSDP}. 
\cref{bSDP} can be further strengthened by incorporating the following so-called \emph{McCormick inequalities}~\cite{McC76}: 
\begin{equation}\label{McCor}
    Y_{ij} \geq 0, \;\; Y_{ij} \geq x_i + x_j -1, \;\; Y_{ij} \leq x_i, \;\; Y_{ij} \leq x_j, \;\; \forall 1 \leq i < j \leq n.
\end{equation}
We then define a stronger convex relaxation for $\QP_n$:
\begin{equation}\label{SDPMC}
\C_n^{\rm SDP+MC} := \Big\{(x, Y) \in \R^{n+\frac{n(n+1)}{2}}: (x,Y) \in \C_n^{\rm SDP}, \; (x,Y)\; {\rm satisfy \; inequalities}~\eqref{McCor}\Big\}.
\end{equation}
In~\cite{AnsBur10}, utilizing ideas from copositive programming, the authors proved that if $n=2$, then $\QP_n = \C_n^{\rm SDP+MC}$, while if $n=3$, then $\QP_n \subsetneq \C_n^{\rm SDP+MC}$. To date, obtaining an explicit characterization of $\QP_3$ remains an open question.

\subsection{Binary quadratic programming and the Boolean quadric polytope} 
Minimizing a quadratic function over the set of binary points, henceforth referred to as binary quadratic programming, is a fundamental $\NP$-hard problem in discrete optimization:
\begin{align}\label[problem]{BQP}
\tag{BQP}
\min \quad & x^\top Q x + c^\top x \\
{\rm s.t.} \quad & x \in \{0,1\}^n. \nonumber
\end{align} 
Since $x^2_i = x_i$ for $x_i \in \{0,1\}$, without loss of generality, we can assume that the diagonal entries of $Q$ are zero. As before, to linearize the objective function we define $Y_{ij} := x_i x_j$ for all $1 \leq i < j \leq n$, and obtain a reformulation of~\cref{BQP} in a lifted space of variables: 
\begin{align}\label[problem]{LBQP}
\tag{$\ell$BQP}
\min \quad & 2\sum_{1\leq i < j \leq n}{q_{ij} Y_{ij}} + c^\top x \\
{\rm s.t.} \quad & Y_{ij} = x_i x_j, \; \forall 1 \leq i < j \leq n, \nonumber\\
& x \in \{0,1\}^n \nonumber.
\end{align} 
In~\cite{Pad89}, Padberg introduced the \emph{Boolean quadric polytope} as the convex hull of the feasible region of~\cref{LBQP}: 
$$
\BQP_n := \conv\Big\{(x, Y) \in \R^{n+\frac{n(n-1)}{2}}: Y_{ij} = x_i x_j, \; \forall 1 \leq i < j \leq n, \; x \in \{0,1\}^n\Big\}.
$$
He then studied the facial structure of $\BQP_n$ and introduced various classes of facet-defining inequalities for it. If $n=2$, then $\BQP_n$ can be fully characterized by McCormick inequalities~\eqref{McCor}. If $n= 3$, then the facet-description of $\BQP_n$ is obtained by adding the following so-called \emph{triangle inequalities}:
\begin{eqnarray}\label{triineq}
    \begin{split}
        & Y_{ij} + Y_{ik} \leq x_i + Y_{jk} \\
        & Y_{ij} + Y_{jk} \leq x_j + Y_{ik} \\
        & Y_{ik} + Y_{jk} \leq x_k + Y_{ij} \\
        & x_i + x_j + x_k - Y_{ij} - Y_{ik} - Y_{jk} \leq 1
    \end{split} \quad \forall 1 \leq i < j < k \leq n,
\end{eqnarray}
to McCormick inequalities. Clearly, $\BQP_n$ is closely related to $\QP_n$. In~\cite{BurLet09}, the authors investigated the connections between $\QP_n$ and $\BQP_n$. In particular, they showed that $\BQP_n$ is the projection of $\QP_n$ obtained by projecting out variables $Y_{ii}$, $i \in [n]$.  
They then showed that if a linear inequality is valid for $\BQP_n$, it is also valid for $\QP_n$.
In addition, their results imply that  both McCormick inequalities and
triangle inequalities define facets of $\QP_n$. Subsequently, they introduced a stronger convex relaxation for $\QP_n$:
\begin{equation}\label{SDPMCTri}
\C_n^{\rm SDP+MC+Tri} := \Big\{(x, Y) \in \R^{n+\frac{n(n+1)}{2}}: (x,Y) \in \C_n^{\rm SDP}, \; (x,Y)\; {\rm satisfy \; inequalities}~\eqref{McCor} \; and \;~\eqref{triineq}\Big\}.
\end{equation}
Yet, they showed that if $n=3$, then $\QP_n \subsetneq \C_n^{\rm SDP+MC+Tri}$. In~\cite{AnsPug25}, the authors introduced additional classes of valid inequalities for $\QP_n$ whose addition to $\C_n^{\rm SDP+MC+Tri}$ results in a stronger convex relaxation for $\QP_n$.

Recall that, given a convex set $\C \subseteq \mathbb{R}^n$, an \emph{extended formulation} for $\C$ is a convex set $\D \subseteq \R^{n + r}$, for some $r > 0$, such that
$\C = \{ x \in \mathbb{R}^n \mid \exists y \in \mathbb{R}^r \text{ such that } (x, y) \in \D \}$.
If $\C$ has a polynomial-size extended formulation $\D$ that can be constructed in polynomial-time, then optimizing a linear function over $\C$ can be done in polynomial-time.
Since~\cref{BQP} is $\NP$-hard in general, unless $\P=\NP$, one cannot obtain, in polynomial time, a polynomial-size extended formulation for $\BQP_n$. 
However, the key observation is that in many applications, the objective function of~\cref{BQP} is \emph{sparse}; \ie $q_{ij} = 0$ for many pairs $i,j$. To this end, Padberg~\cite{Pad89} introduced the Boolean quadric polytope for sparse problems, which we define next. Let $G=(V,E)$ be a graph, where we define a node $i \in V$ for each independent variable $x_i$, $i \in [n]$. Two nodes $i,j$ are adjacent; \ie $\{i,j\} \in E$, when $q_{ij} \neq 0$. We then define the Boolean quadric polytope of a graph $G$ as follows:
$$
\BQP(G) := \conv\Big\{(x, Y) \in \{0,1\}^{V \cup E}: Y_{ij} = x_i x_j, \; \forall \{i,j\} \in E\Big\}.
$$
In~\cite{Pad89}, Padberg obtained sufficient conditions in terms of the structure of graph $G$ under which $\BQP(G)$ admits an extended formulation whose size can be upper bounded by a polynomial in $|V|,|E|$. The first result characterizes the case for which McCormick inequalities~\eqref{McCor} define the Boolean quadric polytope:

\begin{proposition}[\cite{Pad89}]\label{padtree}
Let $G=(V,E)$ be a graph. If $G$ is acyclic, then $\BQP(G)$ is defined by $4|E|$ linear inequalities; \ie inequalities~\eqref{McCor} for all $\{i,j\} \in E$. 
\end{proposition}

Subsequently, Padberg~\cite{Pad89} introduced odd-cycle inequalities, which serve as a generalization of triangle inequalities~\eqref{triineq}. He proved that if $G$ is a series-parallel graph, then the polytope obtained by adding odd-cycle inequalities to McCormick inequalities defines the Boolean quadric polytope. This result implies the following.

\begin{proposition}[\cite{Pad89}]\label{padcycle}
Let $G=(V,E)$ be a graph. If $G$ is a chordless cycle, then $\BQP(G)$ has a polynomial-size linear extended formulation. 
\end{proposition}

More generally, in~\cite{laurent2009sums,KolKou15,BieMun18}, the authors proved that given a graph $G=(V,E)$ with treewidth $\tw(G) = \kappa$, the polytope $\BQP(G)$ has a linear extended formulation with $O(2^{\kappa}|V|)$ variables and inequalities. 
Moreover, from~\cite{ChekChu16,AboFio19} it follows that the linear extension complexity of $\BQP(G)$ grows exponentially with the treewidth of $G$. Recall that given a polytope $\P$, 
the linear extension complexity of $\P$ is the minimum number of linear inequalities and equalities in a linear extended formulation of $\P$. Hence, a bounded treewidth for $G$ is a necessary and sufficient condition for the existence of a polynomial-size extended formulation for $\BQP(G)$.

\subsection{Our contributions}

Motivated by the rich literature on the complexity of the Boolean quadric polytope of sparse graphs, in this paper, we study sparse box-constrained quadratic programs. 

In a similar vein to Padberg~\cite{Pad89}, to exploit the sparsity of $Q$, we introduce a graph representation for~\cref{pQP}.  Consider a graph $G=(V,E, L)$,
where $V$ denotes the node set of $G$, $E$ denotes the edge set of $G$ in which each $\{i,j\} \in E$ contains two distinct nodes $i,j \in V$, and $L$ denotes the loop set of $G$ in which each $\{i,i\} \in L$ is a loop connecting some node $i \in V$ to itself.
We then associate a graph $G$ to~\cref{lQP}, where we define a node $i$ for each independent variable $x_i$, for all $i\in [n]$, two distinct nodes $i$, $j$ are adjacent if the coefficient $q_{ij}$ is nonzero, and there is a loop $\{i,i\}$ for some $i \in [n]$, if the coefficient $q_{ii}$ is nonzero. We say that a loop $\{i,i\}$ for some $i \in [n]$, is a \emph{plus loop}, if $q_{ii} > 0$ and is a \emph{minus loop}, if $q_{ii} < 0$. We denote the set of plus loops by $L^+$, and we denote the set of minus loops by $L^-$. We then have $L = L^- \cup L^+$.
Henceforth, for notational simplicity, we denote variables $x_i$, by $z_i$, for all $i \in [n]$, and we denote variables $Y_{ij}$, by $z_{ij}$, for all $\{i,j\} \in E \cup L$. We then consider the following reformulation of~\cref{lQP} that enables us to exploit its sparsity: 
\begin{align}\label[problem]{lQPG}
\tag{$\ell$QPG}
\min \quad & \sum_{\{i,i\} \in L}{q_{ii} z_{ii}}+2 \sum_{\{i,j\} \in E}{q_{ij} z_{ij}}  + \sum_{i\in V}{c_i z_i} \\
{\rm s.t.} \quad & z_{ii} \geq z^2_i, \; \forall \{i, i\} \in L^+\nonumber\\
& z_{ii} \leq z^2_i, \; \forall \{i, i\} \in L^- \nonumber\\ 
& z_{ij} = z_i z_j, \; \forall \{i,j\} \in E\nonumber\\
& z_i \in [0,1], \; \forall i \in V. \nonumber
\end{align}
Furthermore, we define: 
\begin{align*}
\QP(G) := \conv\Big\{z \in \R^{V\cup E \cup L}: \; & z_{ii} \geq z^2_i, \; \forall \{i, i\} \in L^+, \; z_{ii} \leq z^2_i, \; \forall \{i, i\} \in L^-, \; z_{ij} = z_i z_j, \\
&\forall \{i,j\} \in E, \;  z_i \in [0,1], \forall i \in V\Big\}.
\end{align*}
In this paper, we obtain sufficient conditions in terms of the structure of graph $G$, under which $\QP(G)$ admits a polynomial-size second-order cone (SOC) representable formulation, which can be constructed in polynomial time. The main contributions of this paper are as follows.

\medskip

\begin{enumerate}
    \item In~\cref{sec: convexification}, we propose a new technique to build SOC-representable convex relaxations for $\QP(G)$. This method can be considered as a generalization of the widely used Reformulation Linearization Technique (RLT)~\cite{SheAda90} to continuous quadratic programs. The proposed SOC-representable relaxation is not implied by the existing SDP relaxations of $\QP(G)$; \ie relaxations $\C^{\rm SDP}_n$, $\C^{\rm SDP+MC}_n$, and $\C^{\rm SDP+MC+Tri}_n$. 
\medskip

    \item In~\cref{sec: convexHull}, we investigate the tightness of the proposed relaxation. Let $G=(V,E,L)$ be a graph, and let $V^+$ be the subset of nodes of $G$ with plus loops.
    We prove that if $V^+$ is a \emph{stable set} of $G$, then 
    $\QP(G)$ is SOC-representable. This extended formulation is obtained by combining the proposed convexification technique with a decomposition argument. In particular, this result implies that if $G$ is a complete graph and has one plus loop, then $\QP(G)$ admits a SOC-representable formulation with $O(2^{|V|})$ variables and inequalities.
\medskip

    \item In~\cref{sec: polysize}, we obtain sufficient conditions under which the proposed extended formulation of~\cref{sec: convexHull} is of polynomial size; that is, it can be upper bounded by a polynomial in $|V|$. Roughly speaking, our main result requires that graph $G$ has a tree decomposition with a small width, such that each bag in the tree decomposition contains at most one node with a plus loop, and each node with a plus loop is contained in a small number of bags of the tree (see~\cref{polysize}). This result serves as a generalization of the bounded treewidth result for $\BQP(G)$ to the continuous setting. As corollaries of this result, we obtain generalizations of~\cref{padtree} and~\cref{padcycle} for the continuous case. 
\end{enumerate}

\medskip 

The remainder of this paper is structured as follows. In~\cref{sec: prelim}, we present the preliminary material that we need to obtain our main results. In~\cref{sec: convexification}, we describe our new convexification technique for nonconvex quadratic programs. In~\cref{sec: convexHull}, we examine the tightness of the proposed relaxation. In~\cref{sec: polysize}, we obtain sufficient conditions under which $\QP(G)$ admits a polynomial-size SOC-representable formulation. Finally, in~\cref{sec: conclude} we present some extensions and discuss directions of future research.

\section{Preliminaries}
\label{sec: prelim}

In this section, we present the preliminary material that we will need to prove our main results in subsequent sections.

\subsection{Extended formulations and the multilinear polytope}

In this paper, we derive SOC-representable extended formulations for $\QP(G)$. These extended formulations are obtained by introducing auxiliary variables that are products of more than two independent variables. To this end, we briefly recall the multilinear polytope and hypergraphs~\cite{dPKha17MOR}. 
A \emph{hypergraph} $G$ is a pair $(V,E)$, where $V$ is a finite set of nodes and $E$ is a set of subsets of $V$ of cardinality at least two, called the edges of $G$. 
With any hypergraph $G= (V,E)$, we associate the multilinear polytope $\MP(G)$ defined as:
\begin{equation*} 
 \MP(G):= \conv\Big\{ z \in \{0,1\}^{V \cup E} : z_e = \prod_{i \in e} {z_{i}}, \; \forall e \in E \Big\}.
 \end{equation*}
 We say that $G=(V,E)$ is a \emph{complete} hypergraph, if $E$ contains all subsets of $V$ of cardinality at least two. The multilinear polytope of a complete hypergraph has a simple structure (see the proof of Theorem~1 in~\cite{dPKha18MPA}):

\begin{lemma}[\cite{dPKha18MPA}]\label{lemma: simplex}
    Let $G$ be a complete hypergraph. Then the multilinear polytope $\MP(G)$ is a simplex.
\end{lemma}

 In~\cite{dPKha18SIOPT,dPKha23mMPA}, the authors give a complete characterization of the class of acyclic hypergraphs for which $\MP(G)$ admits a polynomial-size linear extended formulation. A key step in proving these results is to establish sufficient conditions for the \emph{decomposability} of the multilinear polytope; a concept that we define next. Given a hypergraph $G=(V,E)$ and $V' \subseteq V$, the \emph{section hypergraph} of $G$ induced by $V'$ is the hypergraph $G'=(V',E')$, where $E' = \{e \in E : e \subseteq V'\}$.
Given two hypergraphs $G_1=(V_1,E_1)$ and $G_2=(V_2,E_2)$, we denote by $G_1 \cap G_2$ the hypergraph $(V_1 \cap V_2, E_1 \cap E_2)$,
and by $G_1 \cup G_2$, the hypergraph $(V_1 \cup V_2, E_1 \cup E_2)$. 
In the following, we consider a hypergraph $G$, and two distinct section hypergraphs of $G$, denoted by $G_1$ and $G_2$, such that $G_1 \cup G_2 = G$.
We say that $\MP(G)$ is \emph{decomposable} into $\MP(G_1)$ and $\MP(G_2)$ if  the description of $\MP(G_1)$ and the description of $\MP(G_2)$, is the description of $\MP(G)$ without the need for any additional constraints.
The next proposition, which follows from~\cref{lemma: simplex},  provides a sufficient condition for the decomposability of the multilinear polytope.

\begin{proposition}[\cite{dPKha18MPA}]
\label{th decomposition}
Let $G_1$,$G_2$ be section hypergraphs of a hypergraph $G$ such that $G_1 \cup G_2 = G$ and $G_1 \cap G_2$ is a complete hypergraph.
Then, $\MP(G)$ is decomposable into $\MP(G_1)$ and $\MP(G_2)$.
\end{proposition}

 With a slight abuse of notation, in the following, we sometimes replace $z_{\{i,j\}}$ by $z_{ij}$. 
 In our setting, in addition to the multilinear equations $z_e = \prod_{i \in e}{z_i}$ for all $e \in E$, we also have relations of the form $z_{ii} \geq z^2_i$ or $z_{ii} \leq z^2_i$. To take these relations into account, we define hypergraphs with loops; \ie $G=(V,E,L)$, where $V, E$ are the same as those defined for loopless hypergraphs and where $L$ denotes the set of loops of $G$. As before, we partition $L$ as $L = L^- \cup L^+$, where $L^-$ and $ L^+$ denote the sets of minus loops and plus loops of $G$, respectively. With any hypergraph $G= (V,E,L)$, we associate the convex set $\PP(G)$ defined as:
\begin{align} \label{extset}
 \PP(G):= \conv\Big\{ z \in \R^{V \cup E \cup L} : \; & z_{ii} \geq z^2_i, \; \forall \{i, i\} \in L^+, \; z_{ii} \leq z^2_i, \; \forall \{i, i\} \in L^-,\; z_e = \prod_{i \in e} {z_i}, \; \forall e \in E, \nonumber\\
 & z_i \in [0,1], \; \forall i \in V \Big\}.
\end{align}
While the convex hull of an unbounded set is not closed, in general, the following lemma indicates that $\PP(G)$ is a closed set.

\begin{lemma}\label{lemma:closed}
The convex set $\PP(G)$ is closed. 
\end{lemma}

\begin{proof}
Define the sets
\begin{align*}
     \S = \Big\{ z \in \R^{V \cup E \cup L} : & \; z_{ii} \geq z^2_i, \; \forall \{i, i\} \in L^+, \; z_{ii} \leq z^2_i, \; \forall \{i, i\} \in L^-,\; z_e = \prod_{i \in e} {z_i}, \; \forall e \in E, \\ 
     & z_i \in [0,1], \; \forall i \in V \Big\},\\
     \S_0 = \Big\{ z \in \R^{V \cup E \cup L} : & \; z_{ii} = z^2_i, \; \forall \{i, i\} \in L,\; z_e = \prod_{i \in e} {z_i}, \; \forall e \in E, \; 
     z_i \in [0,1], \; \forall i \in V \Big\},\\
     \S_\infty =\Big\{z \in \R^{V \cup E \cup L} : & \; z_{ii} \geq 0, \; \forall \{i,i\} \in L^+, \; z_{ii} \leq 0, \; \forall \{i,i\} \in L^-, \; z_p = 0, \; \forall p \in V \cup E\Big\}.
\end{align*}
We then have $\S = \S_0 \oplus \S_{\infty}$, where $\oplus$ denotes the Minkowski sum of sets. Since the operations of taking the convex hull and the Minkowski sum commute, it follows that $\PP(G)=\conv(\S) = \conv(\S_0) \oplus \S_{\infty}$. Moreover, $\conv(\S_0)$ is a compact convex set and $\S_{\infty}$ is a closed convex cone. This in turn implies that $\PP(G)$ is closed (see corollary~9.1.2 in~\cite{Roc70}).
\end{proof}

The next lemma characterizes the set of extreme points of $\PP(G)$. This result enables us to obtain our extended formulations. In the following, given a set $\C$, we denote by $\bar \C$ the closure of $\C$. Moreover, we denote by ${\rm ext}(\C)$ and ${\rm exp}(\C)$ the set of extreme points and the set of exposed points of $\C$, respectively.

\begin{lemma}\label{extp}
Let $G=(V,E,L)$ be a hypergraph. Then the set
    \begin{align*}
        \Q = \Big\{z \in \R^{V\cup E \cup L}: \; &z_i \in [0,1], \; \forall i \in V: \{i,i\} \in L^+, z_i \in \{0,1\}, \; \forall i \in V: \{i,i\} \notin L^+, \\
        & z_{p} = \prod_{i \in p}{z_i}, \; \forall p \in E \cup L \Big\}, \nonumber
    \end{align*}
    is the set of extreme points of $\PP(G)$.
\end{lemma}

\begin{proof}
In the following we prove that $\Q$ is the set of exposed points of $\PP(G)$. 
By~\cref{lemma:closed}, $\PP(G)$ is a closed convex set; hence, by an application of Straszewicz's Theorem (see theorem 18.6 in~\cite{Roc70}), we get ${\rm ext}(\PP(G)) = \overline{{\rm exp(\PP(G))}} = \bar \Q = \Q = {\rm exp(\PP(G))} \subseteq {\rm ext}(\PP(G))$, where the third equality follows since $\Q$ is a closed set. Therefore, the equality holds throughout and we conclude that ${\rm ext}(\PP(G))= {\rm exp}(\PP(G))$.

 We now prove that $\Q={\rm exp(\PP(G))}$.  First, fix each $z_i$, $i \in V$ with $\{i,i\} \in L^+$ to any value in $[0,1]$. It then follows that a minimizer of any linear function over the resulting set is attained at a binary point. This is because a minimizer of a concave function or a multilinear function over the unit hypercube is attained at a binary point~\cite{tardella1988class}. This shows that ${\rm exp(\PP(G))} \subseteq \Q$.

   Next, we show that any point $\bar z \in \Q$ is an exposed point of $\PP(G)$. To this end, it suffices to construct a linear function whose unique minimizer is attained at $\bar z$.
   Define $V^-:=\{i \in V: \{i,i\} \in L^-\}$ and $V^+:=\{i \in V: \{i,i\} \in L^+\}$. Consider the linear function
   $$
   f= \sum_{i \in V^+}{\big(z_{ii}-2\bar z_i z_i\big)}- \sum_{i \in V^-}{\big(z_{ii}+(\bar z_i-\frac{3}{2})z_i\big)}+\sum_{i \in V \setminus (V^- \cup V^+)}{(1-2 \bar z_i)z_i}.
   $$
   We show that $\bar z$ is the unique minimizer of $f$ over $\PP(G)$.  Notice that, the minimum of $f$ over $\PP(G)$ is attained because the coefficients of all $z_{ii}$, $\{i,i\} \in L^+$ (resp. $\{i,i\} \in L^-$) are positive (resp. negative).
   It then follows that minimizing $f$ over $\PP(G)$
is equivalent to minimizing the following (quadratic) function over $\PP(G)$: 
   $$
   \tilde f= \sum_{i \in V^+}{(z_{i}-\bar z_i)^2}- \sum_{i \in V^-}{\big(z^2_i+(\bar z_i-\frac{3}{2})z_i\big)}+\sum_{i \in V \setminus (V^- \cup V^+)}{(1-2 \bar z_i)z_i}.
   $$
   First consider the first term in $\tilde f$. Clearly, the unique minimizer of this expression is attained at $z_i = \bar z_i$ for all $i \in V^+$. Next, consider the second term in $\tilde f$ for some $i \in V^-$. If $\bar z_i = 0$, then this term simplifies to $\frac{3}{2}z_i - z^2_i$, whose unique minimizer is attained at $z_i = 0$.   If $\bar z_i = 1$, then this term simplifies to $\frac{1}{2}z_i - z^2_i$, whose unique minimizer is attained at $z_i = 1$. Therefore, the unique minimizer of the second term is obtained at $z_i = \bar z_i$ for all $i \in V^-$. Finally, consider the last term for some $i \in V \setminus (V^- \cup V^+)$. If $\bar z_i = 0$, then this term simplifies to  $z_i$ and it unique minimizer is attained at  $z_i = 0$. If $\bar z_i = 1$, then this term simplifies to $-z_i$ and its unique minimizer is attained at $z_i = 1$. Therefore, the unique minimizer of the third term is attained at $z_i = \bar z_i$ for all $i \in V \setminus (V^- \cup V^+)$.

   Thus far, we have proved that at any minimizer of $f$ over $\PP(G)$ we have $z_i = \bar z_i$ for all $i \in V$ and $z_{ii} = z^2_i = \bar z^2_i = \bar z_{ii}$ for all $\{i,i\} \in L$. Moreover, from the definition of $\PP(G)$ it follows that at a minimizer we have $z_e = \prod_{i \in e} {z_i} = \prod_{i \in e} {\bar z_i} = \bar z_e$. Therefore, the unique minimizer of $f$ over $\PP(G)$ is attained at $\bar z$, implying that $\Q \subseteq {\rm exp}(\PP(G))$. Therefore, $\Q = {\rm exp}(\PP(G))$ and this completes the proof.    
   
\end{proof}

The next lemma enables us to obtain an extended formulation for $\QP(G)$ by obtaining the description of $\PP(G')$ for certain hypergraphs $G'$.

\begin{lemma}\label{obsxxx}
      Consider two hypergraphs $G_1=(V,E_1, L)$ and $G_2 = (V, E_2, L)$ such that $E_1 \subseteq E_2$. Then the formulation for $\PP(G_2)$ is an extended formulation for $\PP(G_1)$. 
\end{lemma}

\begin{proof}
From~\cref{extp} it follows that the projection of every extreme point of $\PP(G_2)$ onto the space of $\PP(G_1)$ is an extreme point of $\PP(G_1)$ and every extreme point of $\PP(G_1)$ is obtained this way. Since the operations of taking the convex hull and projection commute, it follows that $\PP(G_1)$ is obtained by projecting out variables $z_e$, $e \in E_2 \setminus E_1$ from the description of $\PP(G_2)$.
\end{proof}

Next, we make use of~\cref{th decomposition} to obtain two similar decomposability results for $\PP(G)$. Together with~\cref{obsxxx}, these results enable us to obtain SOC-representable formulations for $\QP(G)$. 
The proof of~\cref{th decomposition}  makes use of~\cref{lemma: simplex}.
More precisely, if we denote by $z^\cap$ the variables corresponding to nodes and edges both in $G_1$ and in $G_2$, by $z^1$ the variables corresponding to nodes and edges in $G_1$ and not in $G_2$, and by $z^2$ the variables corresponding to nodes and edges in $G_2$ and not in $G_1$, then $\MP(G)$ is decomposable into $\MP(G_1)$ and $\MP(G_2)$ if and only if
\begin{align} 
\label{eq def decomposability}
 (z^1,z^\cap) \in \MP(G_1), \ (z^\cap,z^2) \in 
 \MP(G_2)
 \quad \Rightarrow \quad 
 (z^1,z^\cap,z^2) \in \MP(G).
 \end{align}
Since $\MP(G_1 \cap G_2)$ is a simplex, it follows that $z^\cap$ can be uniquely written as a convex combination of the extreme points of $\MP(G_1 \cap G_2)$. Therefore, condition~\eqref{eq def decomposability} is satisfied.

Our first decomposability result indicates that, to study the facial structure of $\PP(G)$ or $\QP(G)$, we can essentially discard the minus loops. 

\begin{lemma}\label{minusloops}
Consider a hypergraph $G=(V,E, L)$, where $L = L^- \cup L^+$. Define a hypergraph $G'$ obtained from $G$ by removing all minus loops; \ie $G' = (V, E, L^+)$. Then the formulation for $\PP(G)$ is obtained by putting together the formulation for $\PP(G')$ with the following linear inequalities:  
\begin{equation}\label{harmless}
    z_{ii} \leq z_i, \; z_i \in [0,1], \; \forall \{i, i\} \in L^-.
\end{equation}
\end{lemma}

\begin{proof}
Let $j \in V$ such that $\{j, j\} \in L^-$. Define the hypergraph $\tilde G:= (V,E, L \setminus \{j, j\})$. We show that the formulation for $\PP(G)$ is obtained by combining the formulation for $\PP(\tilde G)$ together with linear inequalities $z_{jj} \leq z_{j}$, $0 \leq z_{j} \leq 1$. The proof then follows from a recursive application of this argument.  Define the set $T_{j}=\conv\{(z_j, z_{jj}): z_{jj} \leq z^2_{j}, \; z_{j} \in [0,1]\} = \{(z_{j}, z_{jj}): z_{jj} \leq z_{j}, \; z_{j} \in [0,1]\}$.
Clearly, at an extreme point of $T_{j}$ we have $z_{j} \in \{0,1\}$. By~\cref{extp}, at any extreme point of $\PP(\tilde G)$ and $\PP(G)$ we also have $z_{j} \in \{0,1\}$. Moreover, $z_{j}$ is the only common variable between the two sets $T_{j}$ and $\PP(\tilde G)$. Therefore, the formulation for $\PP(G)$ is obtained by putting together the formulation for $\PP(\tilde G)$ and the formulation for $T_{j}$, and this completes the proof.
\end{proof}

If $L^+ = \emptyset$, then from~\cref{extp} it follows that $\QP(G)$ is a polytope. Moreover, by~\cref{minusloops}, the description of $\QP(G)$ is obtained by putting together the description of the Boolean quadric polytope $\BQP(G')$ together with inequalities~\eqref{harmless}.

Next, we obtain a rather straightforward generalization of~\cref{th decomposition} to the case where hypergraphs $G, G_1, G_2$ also have loops. To this end, we modify the definition of section hypergraph as follows. Given a hypergraph $G=(V,E, L)$, and $V' \subseteq V$, the section hypergraph of $G$ induced by $V'$ is the hypergraph $G'=(V',E', L')$, where $E' = \{e \in E : e \subseteq V'\}$ and $L'=\{\{i,i\} \in L: i \in V'\}$. Moreover, if a loop is a plus (resp. minus) loop in $G$, it is also a plus (resp. minus) loop in $G'$. Given hypergraphs $G_1=(V_1,E_1,L_1)$ and $G_2=(V_2,E_2,L_2)$, we denote by $G_1 \cap G_2$ the hypergraph $(V_1 \cap V_2, E_1 \cap E_2, L_1 \cap L_2)$,
and we denote by $G_1 \cup G_2$, the hypergraph $(V_1 \cup V_2, E_1 \cup E_2, L_1 \cup L_2)$. 
We say that $\PP(G)$ is \emph{decomposable} into $\PP(G_1)$ and $\PP(G_2)$ if the system comprising the description of $\PP(G_1)$ and the description of $\PP(G_2)$, is the description of $\PP(G)$. 
Then the following decomposability result is a direct consequence of~\cref{minusloops} and~\cref{th decomposition}.

\begin{corollary}\label{cor: decomp}
Let $G=(V,E,L)$ be a hypergraph.
Let $G_1$,$G_2$ be section hypergraphs of $G$ such that $G_1 \cup G_2 = G$ and $G_1 \cap G_2$ is a complete hypergraph. Moreover, suppose that the hypergraph $G_1 \cap G_2$ has no plus loops. 
Then, $\PP(G)$ is decomposable into $\PP(G_1)$ and $\PP(G_2)$.    
\end{corollary}

\begin{proof}
Let $L= L^- \cup L^+$, where as before $L^-$ and $L^+$ denote the set of minus loops and plus loops of $G$, respectively. Define the hypergraph $G':=(V,E,L^+)$.
    First, by~\cref{minusloops}, the description of $\PP(G)$ is obtained by putting together the description of $\PP(G')$ together with inequalities~\eqref{harmless}. Let $G_1=(V_1, E_1, L_1)$ with $L_1=L^-_1 \cup L^+_1$ and $G_2=(V_2, E_2, L_2)$ with
    $L_2=L^-_2 \cup L^+_2$. Define the hypergraphs $G'_1 := (V_1, E_1, L^+_1)$
    and $G'_2 := (V_2, E_2, L^+_2)$. Notice that $G'=G'_1 \cup G'_2$.
    Moreover $G'_1 \cap G'_2$ is a complete hypergraph without any loops; by~\cref{lemma: simplex}, $\PP(G'_1 \cap G'_2)$ is a simplex. Therefore, $\PP(G')$ is decomposable into $\PP(G'_1)$ and $\PP(G'_2)$. Together with~\cref{minusloops}, this in turn implies that $\PP(G)$ is decomposable into $\PP(G_1)$ and $\PP(G_2)$.
\end{proof}

\subsection{The RLT  and the multilinear polytope}

The RLT developed by Sherali and Adams~\cite{SheAda90} provides an automated mechanism to construct hierarchies of increasingly stronger linear programming relaxations for binary polynomial programs with $n$ variables,  such that the $n$th level relaxation of the hierarchy coincides with the convex hull of the feasible solutions. Our convexification technique in this paper can be considered as a generalization of RLT for box-constrained nonconvex quadratic programs. 

In the following, we provide a brief overview of the RLT terminology and results that we will use to develop our convexification technique. Let $G=(V,E)$ be a hypergraph and let $\S(G):=\{z \in \{0,1\}^{V\cup E}: z_e = \prod_{i \in e}{z_i}, \forall e \in E\}$. Define $n:=|V|$. In a similar vein to~\cite{SheAda90}, for some $d \in [n]$, we define a \emph{polynomial factor} as:
$$
f_d(J_1, J_2) := \prod_{i \in J_1}{z_i}\prod_{i \in J_2}{(1-z_i)}, \quad J_1, J_2 \subseteq [n], \; J_1 \cap J_2 = \emptyset, \;|J_1 \cup J_2| = d,
$$
where we define $z_{\emptyset}: = 1$.
Since by assumption $z$ is binary valued, it follows that $f_d(J_1, J_2) \geq 0$ for all $d \in [n]$ and all $J_1, J_2$ satisfying the above conditions. We then expand these polynomial factors and rewrite them as: 
$$
f_d(J_1, J_2) = \sum_{t: t\subseteq J_2}{(-1)^{|t|} \prod_{i\in J_1}{z_i}\prod_{i\in t}{z_i}}.
$$
Next, we linearize the polynomial factors 
by introducing new variables $z_{J_1\cup t} :=  \prod_{i\in J_1} {z_i}\prod_{i\in t}{z_i}$ for all $t \subseteq J_2$. Notice that 
$J_1 \cup t$ might be present in $E$ for some $t \subseteq J_2$, in which case no new variable is introduced.
We then obtain the following system of valid linear inequalities (possibly in an extended space) for the set $\S(G)$: 
\begin{equation}\label{rlteq}
\ell_d(J_1, J_2) := \sum_{t: t\subseteq J_2}{(-1)^{|t|} z_{J_1 \cup t}} \geq 0 , \quad J_1, J_2 \subseteq [n], \; J_1 \cap J_2 = \emptyset, \;|J_1 \cup J_2| = d.
\end{equation}

\begin{observation}\label{remark:dmnc}
     In~\cite{SheAda90}, the authors prove that for any $1\leq d < n$,  the system of inequalities $\ell_d(J_1, J_2) \geq 0$ for all $J_1, J_2 \subseteq [n]$ satisfying $J_1 \cap J_2 = \emptyset$ and $|J_1 \cup J_2| = d$ is implied by the system of inequalities $\ell_{d+1}(J_1, J_2) \geq 0$ for all $J_1, J_2 \subseteq [n]$ satisfying $J_1 \cap J_2 = \emptyset$ and $|J_1 \cup J_2| = d+1$. The proof uses the fact that for any $k \in [n] \setminus (J_1 \cup J_2)$ we have
     $$
     f_{d}(J_1, J_2) = f_{d+1}(J_1 \cup \{k\}, J_2)+ f_{d+1}(J_1 , J_2\cup \{k\}),
     $$ 
     implying that the nonnegativity of $\ell_{d}(J_1, J_2)$ follows from the nonnegativity of
     $\ell_{d+1}(J_1 \cup \{k\}, J_2)$ and $\ell_{d+1}(J_1 , J_2\cup \{k\})$
     (see Lemma~1~\cite{SheAda90}). 
\end{observation}

An explicit description for the multilinear polytope of a complete hypergraph can then be obtained using the RLT.
We present this description next. 

\begin{proposition}[\cite{SheAda90}]\label{prop:RLT} 
Let $G=(V,E)$ be a complete hypergraph with $n:=|V|$. Then the multilinear polytope $\MP(G)$ is given by
\begin{equation}\label{eq:rlt}
\ell_n(J, V \setminus J) \geq 0, \quad \forall J \subseteq V,
\end{equation}
where $\ell_n(J, V \setminus J)$ is defined by~\eqref{rlteq}.
\end{proposition}
We are interested in generalizing~\cref{prop:RLT} to the case in which $G$ has plus loops. 

\section{A new convexification scheme}
\label{sec: convexification}

In this section, we propose a new technique to build SOC-representable convex relaxations for $\QP(G)$. Our convexification technique serves as a generalization of RLT for box-constrained quadratic programs. In~\cite{SheTun95} the authors proposed a relatively direct extension of the original RLT to nonconvex quadratic programs.
In~\cite{Hertog21}, the authors present a generalization of RLT, which they refer to as the Reformulation Perspectification Technique, for nonconvex continuous optimization problems.
However, in these studies, it remains unclear whether and under what conditions the resulting relaxations constitute extended formulations for $\QP(G)$. In contrast, as we detail in the next section, our relaxation technique yields an exact characterization of $\QP(G)$ for a large class of graphs $G$.

To formalize our technique, we need to introduce some terminology first. Consider the function $f(u,v) := \frac{u^2}{v}$, $u \in \R$ and $v > 0$. Recall that $f(u,v)$ is a convex function because it is the perspective of the convex function $u^2$. We define the closure of $f(u,v)$, denoted by $\hat f(u,v)$ as follows:
\begin{align}\label{cpersp}
\hat f(u,v) =
\begin{cases}
& \frac{u^2}{v}, \qquad {\rm if}\; v > 0\\
& 0,           \; \;  \qquad     {\rm if} \; u=v=0\\
& +\infty        \qquad      {\rm if} \; u\neq 0, \; v=0.
\end{cases}
\end{align}
For notational simplicity, in the remainder of this paper, whenever we write a function of the form $\frac{u^2}{v}$, or its composition with an affine mapping, we imply its closure as defined by~\eqref{cpersp}.

Consider a graph $G=(V,E, L)$, with $L= L^- \cup L^+$ and suppose that $L^+ \neq \emptyset$. For each plus loop $\{i,i\} \in L^+$, define 
\begin{equation}\label{neigbours}
N(i) := \{j \in V: \{i, j\} \in E \cup L\}.
\end{equation}
Notice that $N(i) \ni i$.  We are now ready to present our new convexification technique. 

\begin{proposition}\label{convRelax}
Let $G=(V, E, L)$ be a graph with $L = L^- \cup L^+$.
Let $\{i,i\} \in L^+$ and let $M \subseteq N(i)$ such that $M \ni i$. Define $d:= |M|$. Then the following inequalities form a closed convex set:
\begin{align}
& z_{ii} \geq\sum_{J \subseteq M: J \ni i} \frac{(\ell_{d}(J,M \setminus J))^2}{\ell_{d-1}(J\setminus \{i\}, M \setminus J)}\label{newineq}\\
& \ell_{d}(J,M \setminus J) \geq 0, \; \forall J \subseteq M,\label{support}
\end{align}
where $\ell_d(\cdot, \cdot)$ is defined by~\eqref{rlteq}.
Moreover, the projection of~\eqref{newineq}-\eqref{support}  onto the space $z_p$, $p \in V \cup E \cup L$ is a convex relaxation of $\QP(G)$.
\end{proposition}

\begin{proof}
For any $J \subseteq M$ such that $J \ni i$, define the function
$$g(M,J):= \frac{(\ell_{d}(J,M \setminus J))^2}{\ell_{d-1}(J\setminus \{i\}, M \setminus J)}.$$ 
First, note that by~\cref{remark:dmnc}, inequalities~\eqref{support} imply that 
$\ell_{d-1}(J\setminus \{i\}, M \setminus J) \geq 0$. It then follows that $g(M,J)$ is a closed convex function because it is obtained by composing the closed convex function $\hat f(u,v)$ defined by~\eqref{cpersp}, with the affine mapping $(u,v) \rightarrow (\ell_{d}(J,M \setminus J), \ell_{d-1}(J\setminus \{i\}, M \setminus J))$. Notice that $g(M,J) \neq +\infty$, because
$\ell_{d-1}(J\setminus \{i\}, M \setminus J) = \ell_{d}(J, M \setminus J) + \ell_{d}(J \setminus \{i\}, (M \setminus J) \cup\{i\})$, which together with inequalities~\eqref{support} implies that whenever $\ell_{d-1}(J\setminus \{i\}, M \setminus J) = 0$, we also have $\ell_{d}(J, M \setminus J) = 0$ and as a result $g(M, J) = 0$. We then deduce that the set defined by inequalities~\eqref{newineq}-\eqref{support} is a closed convex set.

Next, we show the validity of inequalities~\eqref{newineq}-\eqref{support} for $\QP(G)$. 
Inequalities~\eqref{support} are RLT 
inequalities and hence their validity follows. Now consider inequality~\eqref{newineq}. 
Define $M' = M \setminus \{i\}$. We have
\begin{align*}
z^2_i &= \sum_{K \subseteq M'}{z^2_i\prod_{j\in K}{z_j}\prod_{j \in M' \setminus K}(1-z_j)}\\
&= \sum_{K \subseteq M'}{\frac{\Big(\prod_{j\in K \cup\{i\}}{z_j}\prod_{j \in M' \setminus K}(1-z_j)\Big)^2}{\prod_{j\in K}{z_j}\prod_{j \in M' \setminus K}(1-z_j)}}\\
& =\sum_{K \subseteq M'}{\frac{\Big(\ell_{d}(K\cup\{i\},M' \setminus K)\Big)^2}{\ell_{d-1}(K, M'\setminus K)}}\\
& = \sum_{J \subseteq M: J \ni i} \frac{(\ell_{d}(J,M \setminus J))^2}{\ell_{d-1}(J\setminus \{i\}, M \setminus J)},
\end{align*}
where the first equality follows thanks to the following identity: 
$$\sum_{K \subseteq M'}{\prod_{j\in K}{z_j}\prod_{j \in M' \setminus K}(1-z_j)} = 1,$$
and where in the second equality we define $\frac{(0)^2}{0} := 0$.
Finally, using $z_{ii} \geq z_i^2$, inequality~\eqref{newineq} follows.
\end{proof}

The following example illustrates the construction of the proposed relaxations.

\begin{example}
    Let $G=(V,E,L)$ be a graph with $V=\{1,2,3\}$, $E=\{\{1,2\},\{1,3\}, \{2,3\}\}$, and $L=L^+=\{\{1,1\}\}$. In this case, we have $N(1)=V$. Letting $M \subseteq N(1)$ such that $M \ni 1$ and $|M|=2$, we obtain the following valid inequalities for $\QP(G)$:
    \begin{align*}
       & z_{11} \geq \frac{z^2_{12}}{z_2}+\frac{(z_1-z_{12})^2}{1-z_2}, \;z_{12} \geq 0, \; z_1 -z_{12} \geq 0, \; z_2 -z_{12} \geq 0, \;1-z_1-z_2 + z_{12} \geq 0 \\
       & z_{11} \geq \frac{z^2_{13}}{z_3}+\frac{(z_1-z_{13})^2}{1-z_3}, \; z_{13} \geq 0, \; z_1 -z_{13} \geq 0, \; z_3-z_{13} \geq 0, \; 1-z_1-z_3 +z_{13} \geq 0.
    \end{align*}
    Alternatively, letting $M = N(1)$, we obtain the following valid inequalities for $\QP(G)$ in an extended space:
    \begin{align*}
    & z_{11} \geq \frac{(z_{123})^2}{z_{23}}+\frac{(z_{12}-z_{123})^2}{z_2-z_{23}}+\frac{(z_{13}-z_{123})^2}{z_3-z_{23}}+\frac{(z_1-z_{12}-z_{13}+z_{123})^2}{1-z_2-z_3+z_{23}} \\
    & z_{123} \geq 0, \;z_{12}-z_{123} \geq 0, \; z_{13}-z_{123}, \; z_{23}-z_{123} \geq 0, \;  z_1-z_{12}-z_{13}+z_{123} \geq 0\\
    & z_2-z_{12}-z_{23}+z_{123} \geq 0, \; z_3-z_{13}-z_{23}+z_{123} \geq 0, \;  1-z_1-z_2-z_3+z_{12}+z_{13}+z_{23}-z_{123}\geq 0.
    \end{align*}
$\diamond$    
\end{example}

Given a plus loop $\{i,i\} \in L^+$ and the set $N(i)$ defined by~\eqref{neigbours}, there are $2^{|N(i)|-1}$ choices for the set $M$ and hence as many choices for constructing inequalities~\eqref{newineq}-~\eqref{support}. The following proposition provides a dominance relationship among the corresponding set of inequalities.

\begin{proposition}\label{dominance}
Let $G=(V, E, L)$ be a graph with $L = L^- \cup L^+$.
Let $\{i, i\} \in L^+$ and let $M_1, M_2 \subseteq N(i)$ such that $M_1, M_2 \ni i$. If $M_1 \subset M_2$, then the system of inequalities~\eqref{newineq}-~\eqref{support} with $M=M_1$ is implied by the system of inequalities~\eqref{newineq}-~\eqref{support} with $M=M_2$.   
\end{proposition}

\begin{proof}
Without loss of generality, let $M_2 = M_1 \cup \{k\}$ for some $k \in N(i) \setminus M_1$. Define $d:=|M_1|$. By~\cref{remark:dmnc}, inequalities~\eqref{support} with $M=M_1$ are implied by inequalities~\eqref{support} with $M=M_2$. Therefore, to prove the statement, it suffices to show that  
$$
\sum_{J \subseteq M_2: J \ni i} \frac{(\ell_{d+1}(J,M_2 \setminus J))^2}{\ell_{d}(J\setminus \{i\}, M_2 \setminus J)} \geq \sum_{J \subseteq M_1: J \ni i} \frac{(\ell_{d}(J,M_1 \setminus J))^2}{\ell_{d-1}(J\setminus \{i\}, M_1 \setminus J)}.
$$
Substituting $M_2 = M_1 \cup \{k\}$, the above inequality can be equivalently written as:
$$
\sum_{\substack{J \subseteq M_1:\\ J \ni i}} \Big(\frac{(\ell_{d+1}(J,M_1 \cup \{k\} \setminus J))^2}{\ell_{d}(J\setminus \{i\}, M_1 \cup \{k\} \setminus J)}+
 \frac{(\ell_{d+1}(J \cup \{k\},M_1 \setminus J))^2}{\ell_{d}(J \cup \{k\}\setminus \{i\}, M_1 \setminus J)}\Big)\geq
 \sum_{\substack{J \subseteq M_1: \\J \ni i}} \frac{(\ell_{d}(J,M_1 \setminus J))^2}{\ell_{d-1}(J\setminus \{i\}, M_1 \setminus J)}.
$$
To show the validity of the above inequality, it suffices to show that for each $J \subseteq M_1$ satisfying $J \ni i$, we have:
\begin{equation}\label{subadd}
\frac{(\ell_{d+1}(J,M_1 \cup \{k\} \setminus J))^2}{\ell_{d}(J\setminus \{i\}, M_1 \cup \{k\} \setminus J)}+
 \frac{(\ell_{d+1}(J \cup \{k\},M_1 \setminus J))^2}{\ell_{d}(J \cup \{k\}\setminus \{i\}, M_1 \setminus J)}\geq  \frac{(\ell_{d}(J,M_1 \setminus J))^2}{\ell_{d-1}(J\setminus \{i\}, M_1 \setminus J)}.
\end{equation}
Recall that by definition, we have 
\begin{align*}
& \ell_{d}(J,M_1 \setminus J) = \ell_{d+1}(J,M_1 \cup \{k\} \setminus J)+ \ell_{d+1}(J \cup \{k\},M_1 \setminus J)\\
& \ell_{d-1}(J\setminus \{i\}, M_1 \setminus J) = \ell_{d}(J\setminus \{i\}, M_1 \cup \{k\} \setminus J)+\ell_{d}(J \cup \{k\}\setminus \{i\}, M_1 \setminus J).
\end{align*}
That is, to show the validity of inequality~\eqref{subadd}, it suffices to show that the function $\hat f(u,v)$ defined by~\eqref{cpersp} is subadditive. This is indeed true, since $\hat f$ is both convex and positively homogeneous, and this completes the proof.
\end{proof}

By~\cref{dominance}, the strongest relaxation is obtained by letting $M=N(i)$ for all $\{i,i\} \in L^+$. However, it is important to note that the system~\eqref{newineq}-~\eqref{support} consists of $\Theta(2^{|M|})$ variables and inequalities. Hence, to obtain a polynomial-size convex relaxation of $\QP(G)$, we should set $|M| \in O(\log|V|)$. 
In the same spirit as in RLT, we can then define a hierarchy of relaxations, where in the $r$th level relaxation, for each $\{i,i\} \in L^+$ we consider all $M \subseteq N(i)$ with $M \ni i$ such that $|M| = \min\{|N(i)|, \; r\}$. Denoting by $d_{\max}$ the maximum degree of a node with a plus loop in $G$, we have $1 \leq r \leq d_{\max}+1$.

\begin{observation}
By~\cref{dominance}, the weakest type of inequalities~\eqref{newineq}-\eqref{support} is obtained by letting $M=\{i\}$, in which case inequality~\eqref{newineq} simplifies to $z_{ii} \geq z^2_i$ and inequalities~\eqref{support} simplify to $z_i \geq 0$ and $1-z_i \geq 0$. These inequalities are indeed redundant because they are already present in the description of $\QP(G)$. The second weakest type of inequalities~\eqref{newineq}-~\eqref{support} is obtained by letting $M=\{i,j\}$ for some $j \in N(i)$. In this case, inequality~\eqref{newineq} simplifies to:
\begin{align*}
z_{ii} \geq & \frac{z^2_{ij}}{z_j} + \frac{(z_i-z_{ij})^2}{1-z_j}= \frac{z^2_{ij}-2z_iz_jz_{ij}+z^2_i z_j}{z_j(1-z_j)}\\
=& \frac{(z_{ij}-z_i z_j)^2+z^2_i z_j (1-z_j)}{z_j (1-z_j)}
=  z^2_i+\frac{(z_{ij}-z_iz_j)^2}{z_j(1-z_j)},
\end{align*}
and inequalities~\eqref{support} simplify to $z_{ij} \geq 0$, $z_i-z_{ij} \geq 0$, $z_j - z_{ij} \geq 0$, and $1-z_i-z_j+z_{ij} \geq 0$. Therefore, in this case, inequality~\eqref{newineq} dominates inequality $z_{ii} \geq z^2_i$ at any point satisfying $z_{ij} \neq z_i z_j$.
\end{observation}

\begin{observation}\label{SOCrep}
    Inequalities~\eqref{newineq}-~\eqref{support} are SOC-representable. To see this, for each $J \subseteq M$ satisfying $J \ni i$ define a new variable $t(J)$. It then follows that an equivalent reformulation of inequalities~\eqref{newineq}-~\eqref{support} in an extended space of variables is given by:
    \begin{align*}
& z_{ii} \geq\sum_{J \subseteq M: J \ni i} {t(J)}\\
& t(J) \ell_{d-1}(J\setminus \{i\}, M \setminus J) \geq (\ell_{d}(J,M \setminus J))^2, \; \forall J \subseteq M: J \ni i\\
& \ell_{d}(J,M \setminus J) \geq 0, \; \forall J \subseteq M.
\end{align*}
The first and third sets of the above inequalities are linear, while the second set consists of rotated second-order cone inequalities composed with an affine mapping. Again, notice that if we let $|M| \in O(\log|V|)$, then the above system is of polynomial size.
\end{observation}

Let $G=(V, E, L)$ be a graph. We next show that for any plus loop $\{i,i\} \in L^+$, if the degree of node $i$ is larger than one, then the proposed convexification scheme enables us to obtain stronger convex relaxations for $\QP(G)$ than the existing ones.

\begin{proposition}\label{compareSDP}
    Consider a graph $G=(V, E, L)$ with $n:=|V|$, $L = L^-\cup L^+$
    and $L^+ \neq \emptyset$.
    Consider inequalities~\eqref{newineq}-~\eqref{support} for some $\{i,i\} \in L^+$ and let $M \subseteq N(i)$ such that $M \ni i$, where $N(i)$ is defined by~\eqref{neigbours}.
    Denote by $\S^i_M$ the projection of these inequalities onto the space $z_p$, $p \in V \cup E \cup L$. Then we have the following:
    \medskip
    
    \begin{itemize}
        \item [(i)] if $|M|  = 2$, then $\S^i_M$ is implied by the relaxation $\C^{\rm SDP+MC}_n$ defined by~\eqref{SDPMC}.
        
        \item [(ii)] if $|M| > 2$, then $\S^i_M$ is not implied by the relaxation $\C^{\rm SDP+MC+Tri}_n$ defined by~\eqref{SDPMCTri}.
    \end{itemize}
\end{proposition}

\begin{proof}
Part~(i) follows from theorem~2 of~\cite{AnsBur10} in which the authors prove that for $n=2$, we have $\QP_n = \C_n^{\rm SDP+MC}$.  Henceforth, let $n \geq |M| \geq 3$. 
By~\cref{dominance} it suffices to show that $\S^i_M$ for $|M| = 3$ is not implied by $\C^{\rm SDP+MC+Tri}_n$. Without loss of generality, let $i = 1$ and $M=\{1,2,3\}$.
Consider the point:
\begin{eqnarray}\label{point}
    \begin{split}
        & \tilde z_1 = \frac{1}{4}, \; \tilde z_2 = \tilde z_3 = \frac{1}{2}, \; \tilde z_i = 0,\; \forall i \in [n] \setminus \{1,2,3\}\\
        & \tilde z_{11} = \frac{3}{16}, \; \tilde z_{22} = \tilde z_{33}=\frac{1}{2}, \; \tilde z_{ii} = 0, \; \forall i \in [n] \setminus \{1,2,3\}\\
        & \tilde z_{23} = \frac{1}{4}, \; \tilde z_{ij} = 0, \; \forall 1 \leq i < j \leq n, \; (i,j) \neq (2,3).
    \end{split}
\end{eqnarray}
We first show that $\tilde z \in \C^{\rm SDP+MC+Tri}_n$. By direct calculation, it can be verified that $\tilde z$ satisfies both McCormick and triangle inequalities. Moreover, inequalities $z_{ii} \leq z_i$, $i \in [n]$ are satisfied. Hence, to complete the argument, it suffices to show that the following matrix is positive semidefinite:
$$
\A = \begin{bmatrix}
1 & \frac{1}{4} & \frac{1}{2} & \frac{1}{2} \\
\frac{1}{4} & \frac{3}{16} & 0 & 0 \\
\frac{1}{2} & 0 & \frac{1}{2} & \frac{1}{4} \\
\frac{1}{2} & 0 & \frac{1}{4} & \frac{1}{2}
\end{bmatrix}.
$$
To this end, it suffices to factorize $\A$ as $\A=L D L^T$ 
where $L$ is a lower triangular matrix with ones in the diagonal and $D$ is a nonnegative diagonal matrix.
By direct calculation it can be checked that the following 
are valid choices for $L$ and $D$:
$$
L = \begin{bmatrix}
1 & 0 & 0 & 0 \\
\frac{1}{4} & 1 & 0 & 0 \\
\frac{1}{2} & -1 & 1 & 0 \\
\frac{1}{2} & -1 & -1 & 1
\end{bmatrix},
\qquad
D = \begin{bmatrix}
1 & 0 & 0 & 0 \\
0 & \frac{1}{8} & 0 & 0 \\
0 & 0 & \frac{1}{8} & 0 \\
0 & 0 & 0 & 0
\end{bmatrix}.
$$
Therefore, $\tilde z \in \C^{\rm SDP+MC+Tri}_n$.
However, as we show next, $(\tilde z_1, \tilde z_2, \tilde z_3, \tilde z_{12}, \tilde z_{13}, \tilde z_{23})\notin \S^i_M$, implying that $\S^i_M$ is not implied by relaxation $\C^{\rm SDP+MC+Tri}_n$.   

Denote by $(\hat z_1, \hat z_2, \hat z_3, \hat z_{12}, \hat z_{13}, \hat z_{23}, \hat z_{123})$ a point satisfying inequalities~\eqref{support} for $M=\{1,2,3\}$ such that 
$\hat z_1 = \tilde z_1$, $\hat z_2= \tilde z_2$, $\hat z_3 = \tilde z_3$ $\hat z_{12} = \tilde z_{12}$, $ \hat z_{13} = \tilde z_{13}$, $\hat z_{23} = \tilde z_{23}$.
Then substituting $\hat z_{12} = \tilde z_{12} = 0$ in inequalities $\hat z_{12} - \hat z_{123} \geq 0$ and $\hat z_{123} \geq 0$, we obtain $\hat z_{123} = 0$. Now consider inequality~\eqref{newineq} with $i=1$ and $M=\{1,2,3\}$:
\begin{equation}\label{magic}
z_{11} \geq \frac{(z_{123})^2}{z_{23}}+\frac{(z_{12}-z_{123})^2}{z_2-z_{23}}+\frac{(z_{13}-z_{123})^2}{z_3-z_{23}}+\frac{(z_1-z_{12}-z_{13}+z_{123})^2}{1-z_2-z_3+z_{23}}.
\end{equation}
Substituting $\hat z$ in the above inequality we get
$$\frac{3}{16} \not\geq \frac{(\frac{1}{4})^2}{1-\frac{1}{2}-\frac{1}{2}+\frac{1}{4}}=\frac{1}{4}.$$
Therefore, $\hat z$ violates inequality~\eqref{newineq} and this completes the proof.
\end{proof}

Recall that in~\cite{BurLet09}, the authors show that for $n=3$, the relaxation $\C^{\rm SDP+MC+Tri}_n$ strictly contains $\QP_n$.
Interestingly, the proof of~\cref{compareSDP} implies an even stronger result. Namely, when $n=3$, even if we have only one square term $z_{11} = z^2_1$, then the relaxation $\C^{\rm SDP+MC+Tri}_n$ is not tight.

\begin{corollary}
Consider the set:
$$
\S =\Big\{z:
z_{11} = z_1^2, \; z_{ij} = z_i z_j, \; 1 \leq i < j \leq 3, \; 0 \leq z_i \leq 1, \; i \in \{1,2,3\}\Big\}.
$$
Then, $\C^{\rm SDP+MC+Tri}_n$ defined by~\eqref{SDPMCTri} is not an extended formulation for the convex hull of $\S$.
\end{corollary}

\begin{proof}
Consider the point $\hat z$ defined in the proof of~\cref{compareSDP}. 
    The proof follows by noting $\hat z \in \C^{\rm SDP+MC+Tri}_n$, while $\hat z$ violates inequality~\eqref{newineq}, which by~\cref{convRelax} is a valid inequality for the set $\S$.
\end{proof}

\begin{observation}
In~\cite{AnsPug25}, the authors introduced valid inequalities for $\QP_3$, referred to as the \emph{extended triangle inequalities}, which are not implied by $\C_n^{\rm SDP+MC+Tri}$. The first group is given by:
\begin{align}
\begin{split}\label{ETRI1}
   & 2 z_1 + z_{11} - 2 z_{12} - 2 z_{13} + z_{23} \geq 0, \\
   & 2 z_2 - 2 z_{12} + z_{13} + z_{22} - 2 z_{23} \geq 0,\\
   & 2 z_3 + z_{12} - 2 z_{13} -2 z_{23} + z_{33} \geq 0.
   \end{split}
\end{align}
Substituting the point $\tilde z$ defined by~\eqref{point}, in the first inequality we get $\frac{2}{4}+\frac{3}{16}-2(0)-2(0)+\frac{1}{4} = \frac{15}{16} \geq 0$, substituting in the second inequality we get $\frac{2}{2}-2(0)+0+\frac{1}{2}-\frac{2}{4} = 1 \geq 0$, substituting in the third inequality we get $\frac{2}{2}+0-2(0)-\frac{2}{4}+\frac{1}{2} = 1 \geq 0$. Hence $\S^i_M$ with $|M| > 2$ is not implied by inequalities~\eqref{ETRI1}. The second group of inequalities introduced in~\cite{AnsPug25} is given by: 
\begin{align}
\begin{split}\label{ETRI2}
   & 4 z_1 + 4 z_{11} - 4 z_{12} - 4 z_{13} + z_{23} \geq 0, \\
   & 4 z_2 - 4 z_{12} + z_{13} + 4 z_{22} - 4 z_{23} \geq 0,\\
   & 4 z_3 + z_{12} - 4 z_{13} -4 z_{23} +4 z_{33} \geq 0.
   \end{split}
\end{align}   
Substituting the point $\tilde z$ defined by~\eqref{point}, in the first inequality we get $\frac{4}{4} + \frac{12}{16}-4(0)-4(0)+\frac{1}{4} = 2 \geq 0$, substituting in the second inequality we get $\frac{4}{2}-4(0)+0+\frac{4}{2}-\frac{4}{4} = 3 \geq 0$, substituting in the third inequality we get $\frac{4}{2}+0-4(0)-\frac{4}{4}+\frac{4}{2} = 3 \geq 0$. Hence $\S^i_M$ with $|M| > 2$ is not implied by inequalities~\eqref{ETRI2}. The third and last group of linear inequalities introduced in~\cite{AnsPug25} consists of nine inequalities, and by direct calculation it can be checked that the point $\tilde z$ defined by~\eqref{point} satisfies these inequalities as well implying that $\S^i_M$ with $|M| > 2$ is not implied by these inequalities. 
\end{observation}

\begin{observation}
In~\cite{AnsPug25}, the authors introduced a SOC relaxation for $\QP_3$ by introducing one extended variable. Interestingly, this variable corresponds to the variable $z_{123}$ in our setting. Their SOC relaxation consists of linear inequalities $\ell_3(J,\{1,2,3\}\setminus J) \geq 0$ for all $J \subseteq \{1,2,3\}$ together with the following inequalities:
$$z^2_{123} \leq z_{11} z_{23}, \quad z^2_{123} \leq z_{22} z_{13}, \quad z^2_{123} \leq z_{33} z_{12},
$$
and all their switchings (see definition~2 in~\cite{BurLet09} for the definition of switching for $\QP_n$). Consider the first inequality above. First, consider all switchings of this inequality in which the variable $z_1$ is not switched.
These inequalities are given by
\begin{align*}
    & z^2_{123} \leq z_{11} z_{23} \quad \Leftrightarrow \quad  z_{11} \geq \frac{z^2_{123}}{z_{23}}\\
    & (z_{12}-z_{123})^2 \leq z_{11} (z_3-z_{23}) \quad \Leftrightarrow \quad  z_{11} \geq \frac{(z_{12}-z_{123})^2}{z_3-z_{23}}\\
    & (z_{13}-z_{123})^2 \leq z_{11} (z_2-z_{23}) \quad \Leftrightarrow \quad  z_{11} \geq \frac{(z_{13}-z_{123})^2}{z_2-z_{23}}\\
    & (z_1-z_{12}-z_{13}+z_{123})^2 \leq z_{11}(1-z_2-z_3+z_{23}) \quad \Leftrightarrow \quad  z_{11} \geq \frac{(z_1-z_{12}-z_{13}+z_{123})^2}{1-z_2-z_3+z_{23}}. 
\end{align*}
It is then simple to see that the above four inequalities are implied by inequality~\eqref{magic}. Next, consider all switchings of $z^2_{123} \leq z_{11} z_{23}$ in which the variable $z_1$ is switched. Using a similar line of arguments as above, it follows that the resulting four inequalities are implied by the following inequality which is obtained by switching $z_1$ in inequality~\eqref{magic}:
\begin{align}\label{notmagic}
    z_{11}-2z_1+1 \geq &\frac{(z_{23}-z_{123})^2}{z_{23}}+\frac{(z_2-z_{12}-z_{23}+z_{123})^2}{z_2-z_{23}}+\frac{(z_3-z_{13}-z_{23}+z_{123})^2}{z_3-z_{23}}\nonumber\\
    +&\frac{(1-z_1-z_2-z_3+z_{12}+z_{13}+z_{23}-z_{123})^2}{1-z_2-z_3+z_{23}}. 
\end{align}
Finally, we prove that inequality~\eqref{notmagic} and inequality~\eqref{magic} are equivalent. Denote by $R_1$ the right-hand side of inequality~\eqref{magic} and by $R_2$ the right-hand side of inequality~\eqref{notmagic}. To show the equivalence of the two inequalities, it suffices to show that: 
\begin{equation}\label{goal}
R_2-R_1 =  1-2z_1.
\end{equation}
We then have: 
\begin{align*}
    R_2-R_1 &= \frac{(z_{23}-z_{123})^2-(z_{123})^2}{z_{23}}+\frac{(z_2-z_{12}-z_{23}+z_{123})^2-(z_{12}-z_{123})^2}{z_2-z_{23}}\\
    & +\frac{(z_3-z_{13}-z_{23}+z_{123})^2-(z_{13}-z_{123})^2}{z_3-z_{23}}\\
    &+\frac{(1-z_1-z_2-z_3+z_{12}+z_{13}+z_{23}-z_{123})^2-(z_1-z_{12}-z_{13}+z_{123})^2}{1-z_2-z_3+z_{23}}\\
    & = z_{23}-2z_{123}+z_2-2z_{12}-z_{23}+2z_{123}+z_3-2z_{13}-z_{23}+2 z_{123}+1-2z_1-z_2-z_3\\
    &+2z_{12}+2z_{13}+z_{23}-2z_{123}\\
    & =1-2z_1
\end{align*}
Therefore, equation~\eqref{goal} is valid. We then conclude that the SOC relaxation of~\cite{AnsPug25} is implied by inequalities~\eqref{newineq} and~\eqref{support} with $|M| =3$.
\end{observation}

\section{Convex hull characterizations}
\label{sec: convexHull}

In this section, we examine the tightness of our proposed convexification technique.  As we prove shortly, given a graph $G=(V,E,L)$, as long as the plus loops are located on non-adjacent nodes of $G$, the set $\QP(G)$ is SOC-representable. 
We use the following lemma regarding a property of product factors to prove our first convex hull characterization.

\begin{lemma}\label{keylemma}
Let $G=(V,E)$ be a complete hypergraph.
For any $J \subseteq V$, consider inequality~\eqref{eq:rlt}. Let $j \notin V$. Then we have the following:
\medskip
\begin{itemize}
\item [(i)] By replacing each variable $z_p$ in inequality~\eqref{eq:rlt} by the variable $z_{p\cup\{j\}}$ for all $p \in \emptyset \cup V \cup E$, we obtain 
$$\ell_{n+1}(J \cup\{j\}, V \setminus J) \geq 0.$$
\item [(ii)] By replacing the variable $z_p$ in inequality~\eqref{eq:rlt} by the expression $z_p- z_{p\cup\{j\}}$ for all $p \in \emptyset \cup V \cup E$, we obtain $$\ell_{n+1}(J, V \cup\{j\}\setminus J) \geq 0.$$
\end{itemize}
\end{lemma}

\begin{proof}
First, consider part~(i); we have:
$$
f_{n+1}(J \cup \{j\}, V \setminus J) = \sum_{t\subseteq V \setminus J}{(-1)^{|t|} z_{j}\prod_{i\in J}{z_i}\prod_{i\in t}{z_i}}.
$$ 
It then follows that
$\ell_{n+1}(J \cup \{j\}, V \setminus J) = \sum_{t\subseteq V \setminus J}{(-1)^{|t|} z_{\{j\} \cup J \cup t}}$. Therefore, by~\eqref{rlteq}, inequality $\ell_{n+1}(J \cup \{j\}, V \setminus J) \geq 0$ is obtained by replacing each variable $z_p$ by the variable $z_{p\cup\{j\}}$ in inequality $\ell_{n}(J, V \setminus J) \geq 0$. Recall that we define $z_{\emptyset}: = 1$. 

Next, consider part~(ii); we have:
\begin{align*}
   f_{n+1}(J, V\cup \{j\} \setminus J) = & \sum_{t\subseteq V\cup\{j\}\setminus J}{(-1)^{|t|} \prod_{i\in J}{z_i}\prod_{i\in t}{z_i}}\\
   = & \sum_{t\subseteq V \setminus J}{(-1)^{|t|}\prod_{i\in J}{z_i}\prod_{i\in t}{z_i}}+\sum_{t\cup\{j\}: t\subseteq V \setminus J}{(-1)^{|t|+1} z_{j}\prod_{i\in J}{z_i}\prod_{i\in t}{z_i}}\\
   = &\sum_{t\subseteq V \setminus J}{(-1)^{|t|}\Big(\prod_{i\in J}{z_i}\prod_{i\in t}{z_i}-z_{j}\prod_{i\in J}{z_i}\prod_{i\in t}{z_i}\Big)}.
\end{align*}
It then follows that
$$
\ell_{n+1}(J, V \cup \{j\}\setminus J) = \sum_{t\subseteq V \setminus J}{(-1)^{|t|}(z_{J \cup t}-z_{\{j\} \cup J \cup t})}.
$$ 
Therefore, by~\eqref{rlteq}, inequality $\ell_{n+1}(J, V\cup \{j\} \setminus J) \geq 0$ is obtained by replacing each variable $z_p$ by the expression $z_p-z_{p\cup\{j\}}$ in inequality $\ell_{n}(J, V \setminus J) \geq 0$. 
\end{proof}

We are now ready to give our first convex hull characterization.
This result serves as a generalization of~\cref{prop:RLT} to the case where the hypergraph $G$ has one plus loop.

\begin{theorem}\label{convres}
    Let $G=(V,E, L)$ be a complete hypergraph with $n$ nodes.
    Suppose that $L=L^+=\{j, j\}$ for some $j \in V$. Then an explicit description for the set $\PP(G)$ defined by~\eqref{extset} is given by:
    \begin{eqnarray}
    \begin{split}\label{convHull}
      & z_{jj} \geq \sum_{J \subseteq V: J \ni \{j\}}\frac{(\ell_n(J, V\setminus J))^2}{\ell_{n-1}(J\setminus \{j\},V\setminus J)}\\
            & \ell_n(J, V \setminus J) \geq 0,  \qquad \forall J \subseteq V.
      \end{split}
    \end{eqnarray}
\end{theorem}

\begin{proof}
    The proof is by induction on the number of nodes $n$ in $V$. In the base case, \ie $n=1$, the hypergraph $G$ consists of a single node $j$ and a single plus loop $\{j, j\}$ and the set $\PP(G)$ is given by $\PP(G)=\conv\{(z_{j}, z_{jj}): \; z_{jj} \geq z^2_{j},\; z_{j} \in [0,1]\}$. In this case, we have 
    $$
    \ell_1(\{j\}, \emptyset) = z_{j}, \;\;
    \ell_1(\emptyset, \{j\}) = 1-z_{j}, \;\; \ell_0(\{j\}\setminus \{j\}, \emptyset) = 1.
    $$ 
    Therefore, inequalities~\eqref{convHull} simplify to $z_{jj} \geq \frac{z^2_{j}}{1} = z^2_{j}$, $z_{j} \geq 0$, and
    $1-z_{j} \geq 0$, which clearly coincides with $\PP(G)$.

     Henceforth, let $n \geq 2$. Let $k \in V \setminus \{j\}$. 
     Let $\PP^0(G)$ (resp. $\PP^1(G)$) denote the face of $\PP(G)$ defined by $z_{k} = 0$ (resp. $z_{k} = 1$). 
    Since by~\cref{lemma:closed}, $\PP(G)$ is a closed convex set, 
     by~\cref{extp} at every extreme point of $\PP(G)$ we have $z_{k} \in \{0,1\}$ and at every extreme direction of $\PP(G)$ we have $z_k = 0$, we deduce that:
     \begin{equation}\label{convdisj}
      \PP(G) = \conv(\PP^0(G) \cup \PP^1(G)).    
     \end{equation}
     Let $\bar G = (V \setminus \{k\}, \bar E, L)$, be a complete hypergraph; \ie $\bar E$ consists of all subsets of $V \setminus \{k\}$ of cardinality at least two, and let $L=L^+=\{j, j\}$. Since the hypergraph $\bar G$ is complete with a plus loop and with one fewer node than the hypergraph $G$, by the induction hypothesis, an explicit description for $\PP(\bar G)$ is given by:     
    \begin{align*}
      & z_{jj} \geq \sum_{\substack{J \subseteq V \setminus \{k\}:\\ J \ni \{j\}}}\frac{(\ell_{n-1}(J, V\setminus (J\cup \{k\})))^2}{\ell_{n-2}(J\setminus \{j\},V\setminus (J\cup \{k\}))}\\
            & \ell_{n-1}(J, V\setminus (J\cup \{k\})) \geq 0,  \qquad \forall J \subseteq V \setminus \{k\}.
    \end{align*}
Denote by $\bar z$ the vector consisting of $z_v$, $v \in V \setminus \{k\}$ and $z_e$ for all $e \in \bar E \cup L$.
It then follows that
\begin{align*}
& \PP^0(G) = \{z \in \R^{V\cup E \cup L}: z_{k} = 0, \; z_e = 0, \; \forall e \in E: e \ni k, \; \bar z \in \PP(\bar G)\},\\
& \PP^1(G) = \{z \in \R^{V\cup E \cup L}: z_{k} = 1, \; z_e = z_{e \setminus \{k\}}, \; \forall e \in E: e \ni k,\; \bar z \in \PP(\bar G)\}.
\end{align*}
By~\eqref{convdisj}, and using the disjunctive programming technique~\cite{Bal85,Roc70}, it follows that $\PP(G)$ is the projection onto the space of variables $z_v$, $v \in V$, and $z_e$, $e \in E \cup L$ of the system~\eqref{sys1}-~\eqref{sys3}:
\begin{align}
    \begin{split}\label{sys1}
        & \lambda_0 + \lambda_1 = 1, \; \lambda_0 \geq 0, \; \lambda_1 \geq 0 \\
        & z_v = z^0_v + z^1_v, \; \forall v \in V \\
        & z_e = z^0_e + z^1_e, \; \forall e \in E\\
        & z_{jj} = z^0_{jj} + z^1_{jj}
    \end{split}
\end{align}

\begin{align}
    \begin{split}\label{sys2}
     & z^0_{k} = 0 \\
     & z^0_{e} = 0, \; \forall e \in E: e\ni k\\ 
     & \ell^0_{n-1}(J, V\setminus (J\cup \{k\})) \geq 0,  \qquad \forall J \subseteq V \setminus \{k\}\\
     & z^0_{jj} \geq \sum_{\substack{J \subseteq V \setminus \{k\}:\\ J \ni \{j\}}}\frac{(\ell^0_{n-1}(J, V\setminus (J\cup \{k\})))^2}{\ell^0_{n-2}(J\setminus \{j\},V\setminus (J\cup \{k\}))},
    \end{split}
\end{align}
where we define $\ell^0_{n}(J, V\setminus J) := \sum_{t: t\subseteq V \setminus J}{(-1)^{|t|} z^0_{J \cup t}}$, and $z^0_{\emptyset} := \lambda_0$; \ie $\ell^0_{n}(\cdot, \cdot)$ is obtained from $\ell_{n}(\cdot, \cdot)$ by replacing $z_v$ with $z^0_v$ for all $v \in V$, $z_e$ with $z^0_e$ for all $e \in E$ and $z_{\emptyset}$,  with $z^0_{\emptyset}$.
\begin{align}
    \begin{split}\label{sys3}
     & z^1_{k} = \lambda_1\\
     & z^1_{e} = z^1_{e \setminus \{k\}}, \; \forall e \in E: e\ni k\\ 
     & \ell^1_{n-1}(J, V\setminus (J\cup \{k\})) \geq 0,  \qquad \forall J \subseteq V \setminus \{k\}\\
     & z^1_{jj} \geq \sum_{\substack{J \subseteq V \setminus \{k\}:\\ J \ni \{j\}}}\frac{(\ell^1_{n-1}(J, V\setminus (J\cup \{k\})))^2}{\ell^1_{n-2}(J\setminus \{j\},V\setminus \{J\cup \{k\}\})},
    \end{split}
\end{align}
where we define $\ell^1_{n}(J, V\setminus J) := \sum_{t: t\subseteq V \setminus J}{(-1)^{|t|} z^1_{J \cup t}}$, and $z^1_{\emptyset} := \lambda_1$;  \ie $\ell^1_{n}(\cdot, \cdot)$ is obtained from $\ell_{n}(\cdot, \cdot)$ by replacing $z_v$ with $z^1_v$ for all $v \in V$, $z_e$ with $z^1_e$ for all $e \in E$ and $z_{\emptyset}$, with $z^1_{\emptyset}$. Notice that to obtain the above formulation, we made use of the fact that $\ell_n(J, V \setminus J)$ is homogeneous in the variables, $z_{J \cup t}$ for all $t \subseteq V \setminus J$, where we define $z_{\emptyset} = 1$.

\medskip 

In the remainder of the proof we project out variables $\lambda_0, \lambda_1, z^0, z^1$ from the system~\eqref{sys1}-~\eqref{sys3} to obtain the description of $\PP(G)$ in the original space. From $z_{k} = z^0_{k} + z^1_{k}$, $z^0_{k} = 0$, and $z^1_{k} = \lambda_1$ it follows that
\begin{equation}\label{first}
\lambda_0 = 1-z_{k}, \quad \lambda_1 = z_{k}.
\end{equation}
By $z_e = z^0_e + z^1_e$ for all $e \in E$, $z^0_e = 0$ and $z^1_e = z^1_{e\setminus \{k\}}$ for all $e \in E$ such that $e \ni k$, we get:
\begin{align}
  &  z^1_e = z^1_{e \setminus \{k\}} = z_e \quad \forall e \in E: e \ni k\label{second}\\ 
  &  z^0_{e \setminus \{k\}} = z_{e \setminus \{k\}} -z_e \quad \forall e \in E: e \ni k\label{third}
\end{align}
We use equations~\eqref{first} to project out $\lambda_0$ and $\lambda_1$, equations~\eqref{second} to project out $z^1_v$ for all $v \in V \setminus \{k\}$ and for all $z^1_e$ for all $e \in \bar E$, and we use equations~\eqref{third} to project out variables $z^0_v$ for all $v \in V \setminus \{k\}$ and $z^0_e$ for all $e \in \bar E$. By part(i) of~\cref{keylemma}, replacing all variables $z^1_p$ with $z_{p \cup \{k\}}$ in inequality
$\ell^1_{n-1}(J, V\setminus (J\cup \{k\})) \geq 0$, we obtain
$\ell_{n}(J \cup \{k\}, V\setminus (J\cup \{k\})) \geq 0$
for any $J \subseteq V \setminus \{k\}$. Moreover, by part(ii) of~\cref{keylemma}, replacing all variables $z^0_p$ with $z_p-z_{p \cup \{k\}}$ in inequality $\ell^0_{n-1}(J, V\setminus (J\cup \{k\})) \geq 0$, we get $\ell_{n}(J, V\setminus J) \geq 0$ for all $J \subseteq V \setminus \{k\}$.
Notice that inequalities $\ell_{n}(J \cup \{k\}, V\setminus (J\cup \{k\})) \geq 0$
for any $J \subseteq V \setminus \{k\}$ together with inequalities $\ell_{n}(J, V\setminus J) \geq 0$ for all $J \subseteq V \setminus \{k\}$ can be equivalently written as inequalities $\ell_n(J, V \setminus J) \geq 0$ for all $J \subseteq V$.
Finally, we use $z_{jj} = z^0_{jj}+ z^1_{jj}$ to project out $z^0_{jj}$.
Hence, system~\eqref{sys1}-\eqref{sys3} simplifies to:
\begin{align*}
      & \ell_n(J, V \setminus J) \geq 0,  \qquad \forall J \subseteq V\\
      &    z_{jj}-z^1_{jj} \geq  \sum_{\substack{J \subseteq V \setminus \{k\}:\\ J \ni \{j\}}}\frac{(\ell_{n}(J, V\setminus J))^2}{\ell_{n-1}(J\setminus \{j\},V\setminus J)}\\
    & z^1_{jj} \geq \sum_{\substack{J \subseteq V \setminus \{k\}:\\ J \ni j}}\frac{(\ell_{n}(J\cup\{k\}, V\setminus (J\cup \{k\})))^2}{\ell_{n-1}(J\cup\{k\}\setminus \{j\},V\setminus (J \cup \{k\}))}= \sum_{\substack{J \subseteq V:\\ J \supseteq \{j,k\}}}\frac{(\ell_{n}(J, V\setminus J))^2}{\ell_{n-1}(J\setminus \{j\},V\setminus J)}.
\end{align*}
Finally, projecting out the variable $z^1_{jj}$, we obtain: 
\begin{align*}
       \ell_n(J, V \setminus J) & \geq 0,  \qquad \forall J \subseteq V\\
       z_{jj} &\geq  \sum_{\substack{J \subseteq V \setminus \{k\}:\\ J \ni \{j\}}}\frac{(\ell_{n}(J, V\setminus J))^2}{\ell_{n-1}(J\setminus \{j\},V\setminus J)}+  \sum_{\substack{J \subseteq V:\\ J \supseteq \{j,k\}}}\frac{(\ell_{n}(J, V\setminus J))^2}{\ell_{n-1}(J\setminus \{j\},V\setminus J)} \\
       & = \sum_{\substack{J \subseteq V:\\ J \ni \{j\}}}\frac{(\ell_{n}(J, V\setminus J))^2}{\ell_{n-1}(J\setminus \{j\},V\setminus J)}.
\end{align*}
Therefore, the statement follows.
\end{proof}

\begin{observation}
   In~\cref{convres} it is necessary to assume that the complete hypergraph $G$ contains exactly one plus loop. To see this, let us revisit the derivation of the first inequality in~\eqref{convHull}. As explained in the proof of~\cref{convRelax}, this inequality is obtained by multiplying the inequality $z_{jj} \geq z^2_j$ by the identity
   $$\sum_{K \subseteq M \setminus \{i\}}{\prod_{j\in K}{z_j}\prod_{j \in M \setminus (K \cup \{i\})}(1-z_j)} = 1.$$
   That is, the derivation of the first inequality in~\eqref{convHull} relies on considering only one node with a plus loop. Indeed, the well-known result of Burer and Anstreicher~\cite{AnsBur10}, which states that $\QP_2 = \C^{\rm SDP+MC}_2$, implies that if the complete hypergraph $G$ contains two plus loops, then the formulation of $\PP(G)$ must include inequalities involving variables corresponding to both plus loops. Moreover, this result suggests that, in this case, $\PP(G)$ may not be SOC-representable.  
\end{observation}

Thanks to~\cref{obsxxx},~\cref{minusloops}, and~\cref{convres}, we obtain an extended formulation for $\QP(G)$, where $G$ is a complete graph with one plus loop. 

\begin{corollary}\label{completeG}
    Let $G=(V,E,L)$, $L = L^- \cup L^+$ be a complete graph with $L^+=\{j, j\}$ for some $j \in V$. Define $n:= |V|$. Then inequalities~\eqref{harmless} together with inequalities~\eqref{convHull} define an extended formulation for $\QP(G)$.
\end{corollary}

The next example illustrates the application of the proposed extended formulation.

\begin{example}
Let $G=(V,E,L)$ be a graph with $V=\{1,2,3\}$, $E=\{\{1,2\},\{1,3\}, \{2,3\}\}$, $L=L^- \cup L^+$, $L^+=\{\{1,1\}\}$ and $L^-=\{\{2,2\}\}$. In this case, by~\cref{completeG}, an extended formulation for $\QP(G)$ is given by:
    \begin{align*}
    & z_{11} \geq \frac{(z_{123})^2}{z_{23}}+\frac{(z_{12}-z_{123})^2}{z_2-z_{23}}+\frac{(z_{13}-z_{123})^2}{z_3-z_{23}}+\frac{(z_1-z_{12}-z_{13}+z_{123})^2}{1-z_2-z_3+z_{23}} \\
    & z_{123} \geq 0, \; \; z_{12}-z_{123} \geq 0, \;\; z_{13}-z_{123} \geq 0, \;\; z_{23}-z_{123} \geq 0 \\ 
    & z_1-z_{12}-z_{13}+z_{123} \geq 0, \; z_2-z_{12}-z_{23}+z_{123} \geq 0, \; z_3-z_{13}-z_{23}+z_{123} \geq 0\\
    & 1-z_1-z_2-z_3+z_{12}+z_{13}+z_{23}-z_{123} \geq 0\\
    & z_{22} \leq z_2, \; 0 \leq z_2 \leq 1.
    \end{align*}
    $\diamond$
\end{example}


Recall that for a graph, a \emph{stable set} is a subset of nodes such that no two nodes are adjacent. For a hypergraph $G=(V,E)$, we say that $V' \subseteq V$ is a \emph{stable set} of $G$, if there exist no two nodes $u,v \in V'$ such that $u,v \in e$ for some $e \in E$. We should remark that our definition of stable sets for hypergraphs is different from the standard definition. 
For a graph (resp. hypergraph) with loops $G=(V,E,L)$ we define the stable set of $G$ as the stable set of the corresponding loopless graph (resp. hypergraph); \ie $(V,E)$.
Thanks to the decomposability result of~\cref{cor: decomp},  the next theorem implies that if the subset of nodes associated with the plus loops of $G$ is a stable set of $G$, then $\PP(G)$ is SOC-representable. See Figure~\ref{fig1} for an illustration of the proof technique.

\begin{figure}[htbp]
\centering
  \epsfig{figure=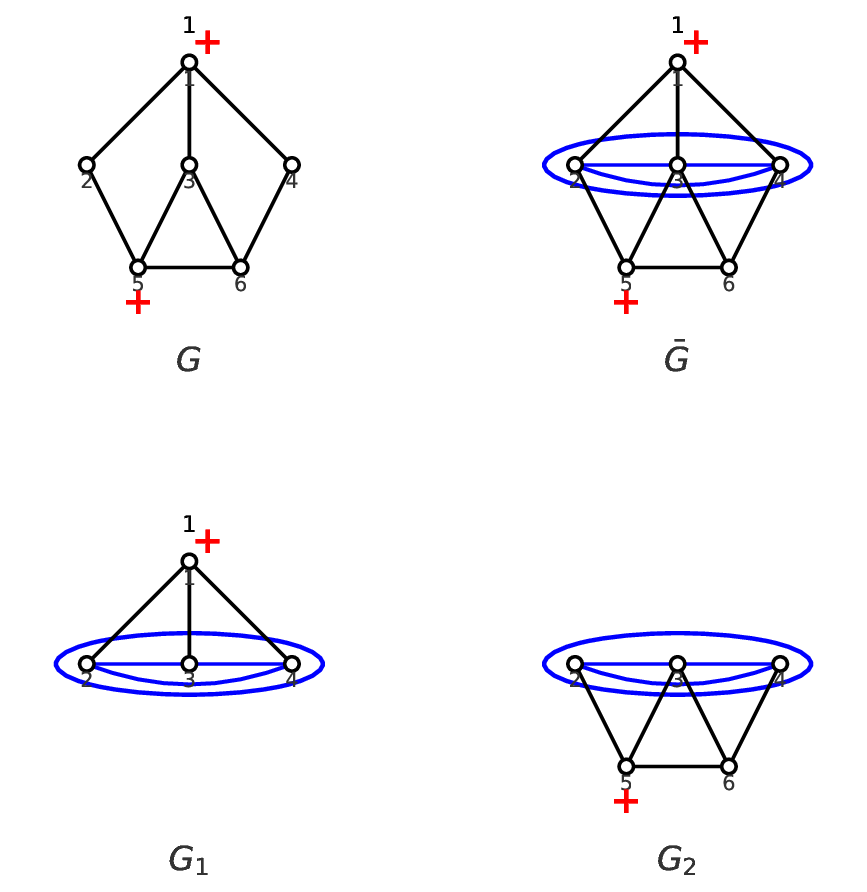, scale=0.5, trim=0mm 0mm 0mm 0mm,clip}
  \caption{An illustration of the proof technique for~\cref{th: stableSet}. Consider the graph $G=(V,E)$; the plus signs on nodes 1 and 5 indicate plus loops on these nodes. Consider node 1; we have $N(1)=\{1,2,3,4\}$. The hypergraph $\bar G$ is constructed by adding to graph $G$ all edges that are subsets of $N(1)\setminus \{1\}$. 
  By~\cref{obsxxx}, the formulation for $\PP(\bar G)$ is an extended formulation for $\QP(G)$. Hypergraph $G_1$ is the section hypergraph of $\bar G$ induced by $N(1)$ and hypergraph $G_2$ is the section hypergraph of $\bar G$ induced by $V \setminus \{1\}$. 
  By~\cref{cor: decomp}, $\PP(\bar G)$ decomposes into $\PP(G_1)$ and $\PP(G_2)$. By~\cref{convHull} and~\cref{obsxxx}, both $\PP(G_1)$ and $\PP(G_2)$ are SOC-representable, implying that $\QP(G)$ is SOC-representable.}
  \label{fig1}
\end{figure}

\begin{theorem}\label{th: stableSet}
    Let $G=(V,E, L)$ be a hypergraph. Define $V^+=\{i \in V: \{i,i\} \in L^+\}$. If $V^+$ is a stable set of $G$, then $\PP(G)$ is SOC-representable.
\end{theorem}

\begin{proof}
    The proof is by induction on the number of nodes in $V^+$. In the base case, we have $|V^+| = 0$; \ie $L= L^-$. Define the loopless hypergraph $G' =(V,E)$. By~\cref{minusloops}, the formulation for $\PP(G)$ is obtained by putting together the formulation for the multilinear polytope $\MP(G')$ together with the linear inequalities~\eqref{harmless}. Therefore, in this case $\PP(G)$  is a polyhedron and the statement follows.

    Now suppose that $|V^+| = k$ for some $k \geq 1$. Let $i \in V^+$, and define 
    $$N(i) := \bigcup_{p \in E \cup L: \; p \ni i}{p}.$$ Moreover, define $N'(i) := N(i) \setminus \{i\}$.  Now, define the hypergraph $\bar G:=(V, \bar E, L)$, where $\bar E = E \cup \{p:\; p \subseteq N'(i), \; |p| \geq 2\}$. By~\cref{obsxxx}, an extended formulation for $\PP(\bar G)$ serves as an extended formulation for $\PP(G)$.
Let $G_1$ be the section hypergraph of $\bar G$ induced by $N(i)$ and let $G_2$ be the section hypergraph of $\bar G$ induced by $V \setminus \{i\}$. By construction $G_1 \cap G_2$ is a complete hypergraph with node set $N'(i)$. Moreover, since $V^+$ is a stable set of $G$ and $i \in V^+$, we have $N'(i) \cap V^+ = \emptyset$; \ie $G_1 \cap G_2$ has no plus loops. Therefore, all the assumptions of~\cref{cor: decomp} are satisfied and $\PP(\bar G)$ is decomposable into $\PP(G_1)$ and $\PP(G_2)$. Since the hypergraph $G_1$ has one plus loop, by~\cref{obsxxx},~\cref{minusloops}, and~\cref{convres}, $\PP(G_1)$ is SOC-representable. Now consider the set $\PP(G_2)$. Note that the hypergraph $G_2$ has one fewer plus loops than the hypergraph $G$. Therefore, by the induction hypothesis $\PP(G_2)$ is SOC-representable as well. We then conclude that $\PP(\bar G)$ and as a result $\PP(G)$ is SOC-representable.
\end{proof}

By~\cref{th: stableSet} and~\cref{obsxxx}, the following result regarding the SOC-representability of $\QP(G)$ is immediate:

\begin{corollary}\label{stableSet}
    Let $G=(V,E, L)$ be a graph. Define $V^+=\{i \in V: \{i,i\} \in L^+\}$. If $V^+$ is a stable set of $G$, then $\QP(G)$ is SOC-representable.
\end{corollary}

\section{Polynomial-size extended formulations}
\label{sec: polysize}

In this section, we obtain sufficient conditions under which $\QP(G)$ has a polynomial-size SOC-representable formulation. Recall that~\cref{stableSet} provides a sufficient condition under which $\QP(G)$ is SOC-representable. However, this representation might be of exponential size. For example, letting $V^+ = \emptyset$ in~\cref{stableSet} and using~\cref{minusloops}, we deduce that in this special case $\QP(G)$ has a polynomial-size extended formulation if and only if the Boolean-quadric polytope $\BQP(G)$ has a polynomial-size extended formulation.
From~\cite{ChekChu16,AboFio19} it follows that the linear extension complexity of $\BQP(G)$ grows exponentially in the treewidth of $G$. Hence, it seems that a bounded treewidth for $G$ is a necessary condition for the polynomial-size representability of $\QP(G)$. In addition, by~\cref{SOCrep}, a necessary condition for polynomial-size representability of our proposed SOC relaxations is that the degree of each node with a plus loop should be $O(\log|V|)$.

In order to formally state our result, we need to introduce some terminology. Given a graph $G = (V, E)$, a \emph{tree decomposition} of $G$ is a pair $(\X, T)$, where $\X = \{X_1, \cdots, X_m\}$ is a family of subsets of $V$, called \emph{bags}, and $T$ is a tree with $m$  nodes,
where each node is associated with bag $X_i$, $i \in [m]$, such that:

\medskip

\begin{enumerate}
    \item $V = \bigcup_{i \in [m]}{X_i}$.
    \item For every edge $\{v_j, v_k\} \in E$, there is a bag $X_i$ for some $i \in [m]$ such that $X_i \ni v_j, v_k$. 
    \item For each node $v \in V$, the set of all bags containing $v$ induces a connected subtree of $T$.
\end{enumerate}
\medskip
The \emph{width} $\omega(\X)$ of a tree decomposition $(\X,T)$ is the size of its largest bag $X_i$ minus one. The \emph{treewidth} $\tw(G)$ of a graph $G$ is the minimum width among all possible tree decompositions of $G$.  In the remainder of this section, without loss of generality, we consider \emph{nonredundant} (or \emph{reduced}) tree decompositions, \ie no bag is contained in any other bag, in which we have $m \leq |V|$~\cite{kloks94}.

Now consider a graph $G=(V,E)$ and let $(\X, T)$ be a tree decomposition of $G$. For each $v \in V$, we define the \emph{spread} of node $v$ with respect to tree decomposition $(\X, T)$, as:
\begin{equation}\label{spread}
    s_v(\X) :=\sum_{i \in [m]: X_i \ni v}(|X_i|-1).
\end{equation}
For any node $v \in V$ that is present only in one bag $X_i$, we have $s_v(\X) = |X_i|-1$. Therefore, for a graph $G$ with $\tw(G) = k$, for any tree decomposition $(\X, T)$ of $G$ we have $s_v (\X) \geq k$ for some $v \in V$.
Moreover, from property 2 of a tree decomposition it follows that the spread of a node is lower bounded by its \emph{degree}; \ie the number of edges incident to $v$. 
If the graph $G$ is a tree, then for any valid tree decomposition of $G$ in which each bag consists of one edge,
the spread of a node $v$ is equal to its degree.

For a graph with loops, \ie $G=(V,E, L)$, we define its tree decomposition, treewidth, and spread, as the tree decomposition, treewidth, and spread of the corresponding loopless graph; \ie $(V,E)$.  Given a hypergraph $G=(V,E,L)$, we define its tree decomposition, treewidth, and spread, as the tree decomposition, treewidth, and spread of its \emph{intersection} graph; \ie the graph $(V,E')$ in which any $v \neq v' \in V$ are adjacent, if $v,v' \in e$ for some $e \in E$. Note that according to this definition, the intersection graph of a graph is the graph itself.

\medskip
We are now ready to state the main result of this section.

\begin{theorem}\label{th:polysize}
    Let $G=(V, E, L)$ be a hypergraph and denote by $(\X,T)$ a tree decomposition of $G$.
    Suppose that $(\X,T)$ satisfies the following properties:
    \medskip
    \begin{itemize}[leftmargin=3em]
        \item [(C1)] For each bag $X_j \in \X$, there exists at most one node $i \in X_j$ such that $\{i,i\} \in L^+$.
        \item [(C2)] The width $\omega(\X)$ is bounded; \ie $\omega(\X) \in O(\log|V|)$.
        \item [(C3)] For each plus loop $\{i,i\} \in L^+$, the spread of node $i$ is bounded; \ie $s_i (\X) \in O(\log|V|)$. 
    \end{itemize}
    \medskip
    Then $\PP(G)$ has a polynomial-size SOC-representable formulation.
\end{theorem}

By~\cref{th:polysize} and~\cref{obsxxx} the following result regarding $\QP(G)$ is immediate:

\begin{corollary}\label{polysize}
    Let $G=(V, E, L)$ be a graph and denote by $(\X,T)$ a tree decomposition of $G$.
    Suppose that $(\X,T)$ satisfies Conditions~(C1)-(C3) of~\cref{th:polysize}.
    Then $\QP(G)$ has a polynomial-size SOC-representable formulation.
\end{corollary}

To prove~\cref{th:polysize}, we make use of the following two lemmata regarding some properties of tree decompositions:

\begin{lemma}\label{bagunify}
 Let $G=(V,E,L)$ be a graph and denote by $(\X,T)$ a tree decomposition of $G$. Suppose that $(\X,T)$ satisfies conditions~(C1)-(C3) of~\cref{th:polysize}. Then $G$ has a tree decomposition $(\X',T')$ satisfying conditions (C1) and (C2) of~\cref{th:polysize} such that each node with a plus loop is present only in one bag.    
\end{lemma}
\begin{proof}
Let $\X = \{X_1, \cdots, X_m\}$ and consider a node $\bar v\in V$ such that $\{\bar v, \bar v\} \in L^+$. Define $Q := \{i \in [m]: X_i \ni \bar v\}$. Denote by $T_{\bar v}$ the connected subtree of $T$ containing node $\bar v$. Now construct $(\X', T')$ as follows. Define a new bag $X_{\bar v}= \cup_{i \in Q} {X_i}$ and let $\X' = (\X  \setminus \{X_i, i \in Q\}) \cup \{X_{\bar v}\}$. We then construct a new tree $T'$ from tree $T$, by first deleting all nodes corresponding to $T_{\bar v}$, then adding a new node $\bar v$ corresponding to the bag $X_{\bar v}$, and connecting $\bar v$ to any node in $T$ that was adjacent to some node in $T_{\bar v}$. It is simple to check that $(\X',T')$ is a valid tree decomposition for $G$. By applying the above argument recursively for each node with a plus loop, we obtain a tree decomposition of $G$ denoted by $(\bar \X, \bar T)$  in which each node with a plus loop is present in one bag.
Notice that since $(\X,T)$ satisfies condition~1 of~\cref{th:polysize}, there is no overlap between the set of bags containing any pair of nodes with plus loops and this property is preserved during the above recursive bag consolidation algorithm.

Now consider the tree decomposition $(\bar \X,\bar T)$ of $G$ in which each node with a plus loop is present in precisely one bag in $\bar \X$. By construction, $(\bar \X,\bar T)$ satisfies condition~1 of~\cref{th:polysize}. Define 
$$s_{\max} (\X) := \max_{i\in V: \{i,i\} \in L^+}{s_i(\X)}.$$
It then follows that $\omega(\bar \X) = \max\{\omega(\X), s_{\max}(\X)\}$. Since $(\bar \X,\bar T)$ satisfies conditions~2 and~3 of~\cref{th:polysize}, we deduce that $\omega(\bar \X) \in O(\log(|V|))$.

\end{proof}

Let $G=(V,E)$ be a graph and let $S \subseteq V$. The graph obtained from $G$ by \emph{removing} nodes in $S$ is the subgraph of $G$ induced by $V \setminus S$; \ie the graph $(V \setminus S, E')$, where $E':=\{\{u,v\} \in E: u,v  \in V \setminus S\}$.

\begin{lemma}\label{auxtree}
    Let $G=(V,E)$ be a graph, and let $(\X, T)$ be a tree decomposition of $G$. 
    Then we have:
    
    \begin{itemize}
        \item [(i)] Let $X_i \in \X$, let $C \subseteq X_i$, and let $G'$ be a graph obtained from $G$ by adding edges $\{u, v\}$ for all $u\neq v \in C$. Then $(\X,T)$ is a tree decomposition of $G'$.
\item [(ii)] Let $X_i \in \X$ be a bag corresponding to a leaf $v_i$ of $T$, and let $T'$ be a subtree of $T$ obtained by removing the leaf $v_i$. Let $G'$ be a graph obtained from $G$ by removing all nodes that appear only in $X_i$. Then $(\X \setminus \{X_i\}, T')$ is a tree decomposition of $G'$.
\end{itemize}
\end{lemma}
\begin{proof}
    First, consider part~(i). Since the node set of $G'$ is identical to that of $G$ and we are not changing the bags, property~1 of a tree decomposition is trivially satisfied. All additional edges in $G'$ are inside the bag $X_i$; therefore, property~2 of a tree decomposition is satisfied.
    Finally, notice that the tree decomposition remains unchanged, implying that property~3 of a tree decomposition is trivially satisfied. 

    Next, consider part~(ii). Let $G'=(V',E')$, where $V' = V \setminus \{v \in X_i: v \notin X_j, \; \forall X_j \in \X, j \neq i\}$. Consider some node $v \in V'$; by definition of $V'$, $v$ must be present in some bag $X_j \in \X$ with $j \neq i$. Since the bag $X_j$ is present in $\X \setminus \{X_i\}$, we deduce that property~1 of a tree decomposition is satisfied.
    Consider some $\{u,v\} \in E'$; since $u \in V'$ and $v \in V'$, we deduce that $u,v \in X_j$ for some bag $X_j \in \X$ with $j \neq i$. Therefore, property~2 of a tree decomposition is satisfied.
    Finally, $T'$ is obtained by removing a leaf from  $T$. Therefore, property~3 of a tree decomposition is satisfied.
\end{proof}

We are now ready to prove~\cref{th:polysize}.
In the following, by $\poly(|V|)$, we imply a polynomial function in $|V|$.

\begin{proof}[Proof of~\cref{th:polysize}]
First consider assumption~(C1); by property~2 of a tree decomposition, (C1) implies that the set $\{i \in V: \{i,i\}\in L^+\}$ is a stable set of $G$. Therefore,~\cref{th: stableSet} implies that $\PP(G)$ is SOC-representable. In the remainder of this proof, we construct a SOC-representable formulation for $\PP(G)$ that is polynomial size.    
Consider a tree decomposition $(\bar \X, \bar T)$ of (the intersection graph) of the hypergraph $G$ satisfying conditions~(C1)-(C3). From~\cref{bagunify} it follows that using $(\bar \X, \bar T)$, we can construct in polynomial time, a tree decomposition $(\X,T)$ of $G$ satisfying conditions~(C1) and~(C2) such that each node with a plus loop is present in only one bag of $\X$. Consider the tree decomposition $(\X, T)$.  
Without loss of generality, assume that $T$ is a rooted tree with any node chosen as its root. Let $X_k$ be a bag corresponding to a leaf of $T$ and denote by $X_j$ the bag whose node is the parent of the node corresponding to $X_k$. 
The following cases arise:
\medskip
\begin{itemize}
    \item [(I)] There exists no node $i \in X_k$ such that $\{i,i\} \in L^+$. Define $X_{\cap} = X_j \cap X_k$. Define the hypergraph $\bar G=(V,\bar E, L)$, where $\bar E = E \cup \{p \subseteq X_{\cap}: |p| \geq 2\}$. By~\cref{obsxxx}, a formulation for $\PP(\bar G)$ serves as an extended formulation for $\PP(G)$. Let $G_1$ be the section hypergraph of $\bar G$ induced by $X_k$, and let $G_2$ be the section hypergraph of $\bar G$ induced by $(V \setminus X_k) \cup X_{\cap}$. 
    By construction, $G_1 \cap G_2$ is a complete hypergraph. Moreover, from property~2 of a tree decomposition it follows that $\bar G = G_1 \cup G_2$. By assumption, there exists no node $i \in X_\cap$ such that $\{i,i\} \in L^+$.  Therefore, by~\cref{cor: decomp}, the set $\PP(\bar G)$ decomposes into $\PP(G_1)$ and $\PP(G_2)$. Now, consider $\PP(G_1)$; notice that $G_1$ has $n_1: = |X_k|$ nodes. Let $G'_1$ be the hypergraph obtained from $G_1$ by removing the loops. Since $G_1$ has no plus loops, by~\cref{minusloops}, an extended formulation for $\PP(G_1)$ is given by an extended formulation for the multilinear polytope $\MP(G'_1)$ together with at most $n_1$ linear inequalities of the form~\eqref{harmless}. 
    An extended formulation for $\MP(G'_1)$ with $2^{n_1}$ variables and inequalities is given by~\cref{prop:RLT}. By definition of the width, we have  $n_1 \leq \omega(\X)$.
    Therefore, from assumption~(C2) it follows that $\PP(G_1)$ has a linear extended formulation with at most $\poly(|V|)$ variables and inequalities. 
    
\medskip

    \item [(II)] There exists a node $i \in X_k$ such that $\{i,i\} \in L^+$. In this case, by construction, $X_k$ is the only bag containing node $i$. This in turn implies that $i \notin X_{\cap}$, where $X_{\cap}$ is as defined in part~(I).
    Moreover, by assumption~(C1), there exists no other node $j \in X_k$ such that $\{j,j\} \in L^+$. Define the hypergraph $\bar G$ as in Part~(I). Again, let $G_1$ be the section hypergraph of $\bar G$ induced by $X_k$, and let $G_2$ be the section hypergraph of $\bar G$ induced by $(V \setminus X_k) \cup X_{\cap}$. Using a similar line of arguments as in Part~(I) above we deduce that $\PP(\bar G)$ decomposes into $\PP(G_1)$
    and $\PP(G_2)$. Now consider $\PP(G_1)$ and denote by $G'_1$ the hypergraph obtained from $G_1$ by removing all its minus loops. By~\cref{minusloops} an extended formulation for $\PP(G_1)$ is obtained by putting together an extended formulation for $\PP(G'_1)$ together with $n_1:=|X_k|$ linear inequalities of the form~\eqref{harmless}. Finally, consider $\PP(G'_1)$; notice that $G'_1$ has one plus loop. Therefore, from~\cref{convres} and~\cref{obsxxx} it follows that $\PP(G'_1)$ is SOC-representable and this extended formulation contains at most $2^{n_1}$ variables and inequalities.
    By definition of the width, we have  $n_1 \leq \omega(\X)$.
    Therefore, from assumption~(C2) it follows that $\PP(G_1)$ has a SOC-representable formulation with at most $\poly(|V|)$ variables and inequalities. 
\end{itemize}
\medskip
Now consider the hypergraph $G_2$ constructed in Parts~(I) and~(II) above. Notice that $G_2$ has at least one fewer node than the hypergraph $G$. By~\cref{auxtree}, a tree decomposition of the hypergraph $G_2$ is given by $(\X\setminus\{X_k\}, T_2)$, where $T_2$ is a subtree of $T$ obtained by removing the leaf corresponding to the bag $X_k$. Clearly, this tree decomposition satisfies conditions~(C1)-(C2) and each node with a plus loop is contained in only one bag. Therefore, the proof follows by induction on the number of nodes of $G$.
\end{proof}

Let $G=(V,E,L)$ be a graph. From the proof of~\cref{polysize} it follows that if a tree decomposition of $G$ that satisfies conditions~(C1)-(C3) is given, then a polynomial-size SOC representable formulation of $\QP(G)$ can be constructed in polynomial time.  A natural question is whether it is possible to check in polynomial-time whether $G$ has a tree decomposition satisfying conditions~(C1)-(C3). It is well-known that condition~(C2) can be checked in polynomial time~\cite{bod93}. However, such a tree decomposition may not satisfy conditions~(C1) and~(C3). We leave as an open question the complexity of checking conditions~(C1)-(C3) of~\cref{th:polysize}.

\medskip

The next three propositions are consequences of~\cref{polysize}. In all cases, given a graph $G$, a polynomial-size SOC representable formulation of $\QP(G)$ can be constructed in polynomial time.
The first result can be considered as a generalization of~\cref{padtree} to the continuous case: 

\begin{proposition}\label{treepoly}
    Let $G=(V, E, L)$ be a graph, and let $V^+ :=\{i \in V: \{i,i\} \in L^+\}$ be a stable set of $G$. Suppose that $(V,E)$ is acyclic and the degree of each node $i \in V^+$ is $O(\log|V|)$.
    Then $\QP(G)$ has a polynomial-size SOC-representable formulation.
\end{proposition}

\begin{proof}
Without loss of generality, assume that $G$ is a tree. Otherwise, we can apply the following argument to each connected component of $G$ separately.
Consider a natural tree decomposition of $G$, where we create a bag for each edge $\{i,j\} \in E$, and the bags are connected according to the topology of the tree $(V,E)$. First, since by assumption $V^+$ is a stable set of $G$, each bag contains at most one node $i$ such that $\{i,i\} \in L^+$, hence satisfying condition~(C1) of~\cref{polysize}.
Second, the width of this tree decomposition is equal to one, hence satisfying condition~(C2) of~\cref{polysize}. Third, since all bags have size two, the spread of each node equals its degree. Since by assumption for each node $i \in V^+$, the degree is $O(\log|V|)$, we deduce that condition~(C3) of~\cref{polysize} is satisfied. Therefore, all the conditions of~\cref{polysize} are satisfied, implying that $\QP(G)$ has a polynomial-size SOC-representable formulation. 
\end{proof}

The next result can be considered as a generalization of~\cref{padcycle} to the continuous case:

\begin{proposition}\label{Cyclepoly}
Let $G=(V,E,L)$ be a graph, and let $V^+ :=\{i \in V: \{i,i\} \in L^+\}$ be a stable set of $G$. Suppose that $(V,E)$ is a chordless cycle on V.  Then $\QP(G)$ has a polynomial-size SOC-representable formulation.
\end{proposition}

\begin{proof}
    Define $n:=|V|$. Let $V=\{v_1, \cdots, v_n\}$, and $E=\{\{v_i, v_{i+1}\}, \forall i \in [n]\}$, where we define $v_{n+1}:= v_1$. Suppose that
    $v_n \notin V^+$. This assumption is without loss of generality because $V^+$ is a stable set of $G$ and $G$ consists of a chordless cycle. Define the bags $X_i:=\{v_i, v_{i+1}, v_n\}$ for all $i \in [n-2]$. Connect the bags so that they form a path, \ie $X_{i}$ is adjacent to $X_{i+1}$ for all $i \in [n-1]$. It is simple to check that this is a valid tree decomposition for $G$. First, consider a bag $X_i:=\{v_i, v_{i+1}, v_n\}$ for some $i \in [n-2]$. By assumption, $v_n \notin V^+$. Moreover, since $V^+$ is a stable set of $G$,
    and $\{v_i, v_{i+1}\} \in E$, at most one of the two nodes $v_i$ and $v_{i+1}$ are in $V^+$. Therefore, each bag contains at most one node in $V^+$, implying that condition~(C1) of~\cref{polysize} is satisfied.  
    Second, the width of this tree decomposition is two, hence, condition~(C2) of~\cref{polysize} is satisfied. Third, nodes $v_1$ and $v_{n-1}$ are each contained in one bag, while nodes $v_2,\cdots, v_{n-2}$ are each contained in two bags. Since each bag has size three, it follows that the spread of nodes $v_1$ and $v_{n-1}$ is two, while the spread of nodes 
    $v_2,\cdots, v_{n-2}$ is four. Notice that the spread of $v_n$ is $2(n-2)$, however this node is not in $V^+$. Therefore, condition~(C3) of~\cref{polysize} is satisfied.
    We then conclude that all conditions of~\cref{polysize} are satisfied, implying that $\QP(G)$ has a polynomial-size SOC-representable formulation.
\end{proof}

The next result indicates that if the set of nodes with plus loops is a subset of a \emph{large} stable set of $G$, then $\QP(G)$ has a polynomial-size SOC-representable formulation. Recall that a large stable set is equivalent to a small vertex cover, and the latter problem is fixed-parameter tractable.

\begin{proposition}\label{Coverpoly}
Let $G=(V,E,L)$ be a graph and let $V^+:=\{i \in V: \{i,i\} \in L^+\} \subseteq S$, where $S$ is a stable set of $G$  such that $|V \setminus S| \in O(\log|V|)$. Then $\QP(G)$ has a polynomial-size SOC-representable formulation.
\end{proposition}

\begin{proof}
Let $S=\{v_1, \cdots, v_m\}$ and denote by $N(v_i)$ the set of nodes of $G$ adjacent to $v_i$ for all $i \in [m]$. Notice that $N(v_i) \subseteq V \setminus S$. Let $(\X,T)$ be a tree decomposition of $G$, where $X_i = \{v_i\} \cup N(v_i)$ for all $i \in [m]$ and $X_{m+1} = V \setminus S$. Create a star-shaped tree by connecting $X_i$, $i \in [m]$ to $X_{m+1}$. From the definition of a stable set, it follows that this is a valid tree decomposition of $G$. Next, we show that $(\X,T)$ satisfies the assumptions of ~\cref{polysize}. First, each bag $X_{i}$, $i \in [m]$ contains at most one node in $V^+$ because $N(v_i) \cap V^+ = \emptyset$. Therefore, condition~(C1) of~\cref{polysize} is satisfied. Second, the width of $(\X,T)$ is at most $|V \setminus S|$, which by assumption is $O(\log|V|)$. Therefore, condition~(C2) of~\cref{polysize} is satisfied. Third, each node $v_i \in V^+$ is present in one bag $X_i$ of $(\X,T)$, and the size of this bag is $N(v_i)+1$, implying that $s_{v_i} (\X)\in O(\log|V|)$. Therefore, condition~(C3) of~\cref{polysize} is satisfied. We then conclude that all conditions of~\cref{polysize} are satisfied, implying that $\QP(G)$ has a polynomial-size SOC-representable formulation.
\end{proof}

For example, suppose that $G=(V,E,L)$ is a bipartite graph and denote by $U, W$ the bipartition of $V$. Suppose that $V^+ \subseteq U$ and that $|W| \in O(\log|V|)$. Then all the assumptions of~\cref{Coverpoly} are satisfied and, therefore, $\QP(G)$ has a polynomial-size SOC-representable extended formulation. 

\section{Extensions and open questions}
\label{sec: conclude}

The techniques used to obtain the convex hull results in this paper easily extend to the case where we replace the quadratic functions corresponding to loops with general convex functions. More precisely, consider a hypergraph $G=(V,E,L)$. For each $i \in V$, let $h_i: \mathbb{R} \rightarrow \mathbb{R}$ be a convex function, and define the set:
\begin{align*} 
 \PP(G_h):= \conv\Big\{ z \in \R^{V \cup E \cup L} : \; & z_{ii} \geq h_i(z_i), \; \forall \{i, i\} \in L^+, \; z_{ii} \leq h_i(z_i), \; \forall \{i, i\} \in L^-, \; z_e = \prod_{i \in e} {z_i},  \\
 &\; \forall e \in E, \; z_i \in [0,1], \; \forall i \in V \Big\}.
\end{align*}
Suppose that $\{i,i\} \in L^+$ for some $i \in V$. 
Denote by $\mathfrak{h}_i(u,v)$ the closure of the perspective of the convex function $h_i(v)$. 
Let $M \subseteq N(i)$ such that $M \ni i$ and let $M' = M\setminus \{i\}$. We then have:
\begin{align*}
z_{ii} \geq h_i(z_i) &= \sum_{K \subseteq M'}{h_i(z_i)\prod_{j\in K}{z_j}\prod_{j \in M' \setminus K}(1-z_j)}\\
&= \sum_{K \subseteq M'}
h_i\left(\frac{z_i\cdot \prod_{j\in K}{z_j}\prod_{j \in M' \setminus K}(1-z_j)}{\prod_{j\in K}{z_j}\prod_{j \in M' \setminus K}(1-z_j)}\right)\cdot \left(\prod_{j\in K}{z_j}\prod_{j \in M' \setminus K}(1-z_j)\right)\\
& =\sum_{K \subseteq M'}
{\mathfrak{h}_i\left(\ell_{d}(K\cup\{i\},M'\setminus K), \; \ell_{d-1}(K, M'\setminus K)\right)}\\
&= \sum_{J \subseteq M: J \ni i}{\mathfrak{h}_i\left(\ell_{d}(J,M \setminus J),\; \ell_{d-1}(J \setminus \{i\}, M\setminus J)\right)}.
\end{align*}
The above inequality together with the inequalities $\ell_{d}(J,M \setminus J) \geq 0$ for all $J \subseteq M$ defines a convex set because $\mathfrak{h}_i$ is a convex function. Moreover, using arguments similar to those in Proposition~\ref{convRelax}, one can verify that the above inequality is well-defined; in particular, the right-hand-side never attains the value of $+\infty$. Furthermore, for the above class of inequalities, the results of Proposition~\ref{dominance} and Theorem~\ref{convres} extend directly to this more general setting by using exactly the same proofs. Finally, the framework can also be generalized to the case where the nodes corresponding to $L^+$ form a stable set in a more general hypergraph $G = (V,E,L)$. 

Several questions remain open. First, we pose an open question whether the stable set assumption on the nodes with plus loops is necessary for SOC-representability of $\QP(G)$.
It is interesting to investigate whether Theorem~\ref{polysize} is tight with respect to the notion of spread; in particular, whether the number of variables required in any valid SOC extended formulation of the convex hull of $\PP(G)$ must necessarily grow exponentially with the spread of nodes with plus loops. Moreover, observe that characterizing $\QP(G)$ is a significantly more general task than solving~\cref{pQP}. For example, the techniques in ~\cite{AsterDey20} give an SOC-based approach to solving~\cref{pQP}, but cannot be used to characterize $\QP(G)$. 
Whether SOC formulations for solving~\cref{pQP} need representations comparable to the size of $\QP(G)$ remains an open question. Finally, incorporating the proposed relaxations into the state-of-the-art mixed-integer nonlinear programming solvers~\cite{IdaNick18} and performing an extensive computational study is a topic of future research.  

\bigskip
\noindent
\textbf{Acknowledgments:} The authors thank the two anonymous referees whose suggestions improved the quality of the paper.

\bigskip
\noindent
\textbf{Funding:}
A. Khajavirad is supported in part by AFOSR grant FA9550-23-1-0123 and ONR grant N00014-25-1-2491.
Any opinions, findings, conclusions or recommendations expressed in this material are those of the authors and do not necessarily reflect the views of the Air Force Office of Scientific Research or the Office of Naval Research.

\bibliographystyle{plain}
\bibliography{biblio}

\begin{thebibliography}{10}

\bibitem{AboFio19}
P.~Aboulker, S.~Fiorini, T.~Huynh, M.~Macchia, and J.~Seif.
\newblock Extension complexity of the correlation polytope.
\newblock {\em Operations Research Letters}, 47(1):47--51, 2019.

\bibitem{AnsPug25}
K.~M. Anstreicher and D.~Puges.
\newblock Extended triangle inequalities for nonconvex box-constrained
  quadratic programming.
\newblock {\em arXiv:2501.09150}, 2025.

\bibitem{AnsBur10}
K.M. Anstreicher and S.~Burer.
\newblock Computable representations for convex hulls of low-dimensional
  quadratic forms.
\newblock {\em Mathematical Programming}, 124:33–--43, 2010.

\bibitem{Bal85}
E.~Balas.
\newblock Disjunctive programming and a hierarchy of relaxations for discrete
  optimization problems.
\newblock {\em SIAM Journal on Algebraic and Discrete Methods}, 6:466--486,
  1985.

\bibitem{BaoSahTaw11}
X.~Bao, N.V. Sahinidis, and M.~Tawarmalani.
\newblock Semidefinite relaxations for quadratically constrained quadratic
  programming: A review and comparisons.
\newblock {\em Mathematical Programming}, 129:129--157, 2011.

\bibitem{BieMun18}
D.~Bienstock and G.~Mu\~noz.
\newblock {LP} formulations for polynomial optimization problems.
\newblock {\em SIAM Journal on Optimization}, 28(2):1121--1150, 2018.

\bibitem{bod93}
H.~L. Bodlaender.
\newblock A linear time algorithm for finding tree-decompositions of small
  treewidth.
\newblock In {\em Proceedings of the twenty-fifth annual ACM symposium on
  Theory of computing}, pages 226--234, 1993.

\bibitem{BurLet09}
S.~Burer and A.~N. Letchford.
\newblock On nonconvex quadratic programming with box constraints.
\newblock {\em SIAM Journal on Optimization}, 20(2):1073--1089, 2009.

\bibitem{ChekChu16}
C.~Chekuri and J.~Chuzhoy.
\newblock Polynomial bounds for the grid-minor theorem.
\newblock {\em Journal of the ACM (JACM)}, 63(5):1--65, 2016.

\bibitem{dPKha17MOR}
A.~Del~Pia and A.~Khajavirad.
\newblock A polyhedral study of binary polynomial programs.
\newblock {\em Mathematics of Operations Research}, 42(2):389--410, 2017.

\bibitem{dPKha18SIOPT}
A.~Del~Pia and A.~Khajavirad.
\newblock The multilinear polytope for acyclic hypergraphs.
\newblock {\em SIAM Journal on Optimization}, 28(2):1049--1076, 2018.

\bibitem{dPKha18MPA}
A.~Del~Pia and A.~Khajavirad.
\newblock On decomposability of multilinear sets.
\newblock {\em Mathematical Programming, Series A}, 170(2):387--415, 2018.

\bibitem{dPKha23mMPA}
A.~Del~Pia and A.~Khajavirad.
\newblock A polynomial-size extended formulation for the multilinear polytope
  of beta-acyclic hypergraphs.
\newblock {\em Mathematical Programming Series A}, pages 1--33, 2023.

\bibitem{HorTuy96}
R.~Horst and H.~Tuy.
\newblock {\em Global Optimization, Deterministic Approaches}.
\newblock Springer, 1996.

\bibitem{IdaNick18}
A.~Khajavirad and N.~V. Sahinidis.
\newblock {A hybrid {LP/NLP} paradigm for global optimization relaxations}.
\newblock {\em Mathematical Programming Computation}, 10(3):383--421, May 2018.

\bibitem{KimKoj03}
S.~Kim and M.~Kojima.
\newblock Exact solutions of some nonconvex quadratic optimization problems via
  {SDP} and {SOCP} relaxations.
\newblock {\em Computational optimization and applications}, 26(2):143--154,
  2003.

\bibitem{kloks94}
T.~Kloks.
\newblock {\em Treewidth: computations and approximations}.
\newblock Springer, 1994.

\bibitem{KolKou15}
P.~Kolman and M.~Kouteck{\`y}.
\newblock Extended formulation for {CSP} that is compact for instances of
  bounded treewidth.
\newblock {\em The Electronic Journal of Combinatorics}, pages P4--30, 2015.

\bibitem{laurent2009sums}
M.~Laurent.
\newblock {Sums of squares, moment matrices and optimization over polynomials}.
\newblock {\em {\em In M. Putinar and S. Sullivant (eds.)}, Emerging
  applications of algebraic geometry, The IMA Volumes in Mathematics and its
  Applications, vol 149 {\em Springer, New York, NY}}, pages 157--270, 2009.

\bibitem{Ye10}
Z.~Luo, W.~Ma, A.~M. So, Y.~Ye, and S.~Zhang.
\newblock Semidefinite relaxation of quadratic optimization problems.
\newblock {\em IEEE Signal Processing Magazine}, 27(3):20--34, 2010.

\bibitem{McC76}
G.P. McCormick.
\newblock Computability of global solutions to factorable nonconvex programs:
  Part {I}: convex underestimating problems.
\newblock {\em Mathematical Programming}, 10:147--175, 1976.

\bibitem{Pad89}
M.~Padberg.
\newblock The {B}oolean quadric polytope: Some characteristics, facets and
  relatives.
\newblock {\em Mathematical Programming}, 45(1--3):139--172, 1989.

\bibitem{Roc70}
R.T. Rockafellar.
\newblock {\em Convex Analysis}.
\newblock Princeton University Press, Princeton, 1970.

\bibitem{AsterDey20}
A.~Santana and S.~S. Dey.
\newblock The convex hull of a quadratic constraint over a polytope.
\newblock {\em SIAM Journal on Optimization}, 30(4):2983--2997, 2020.

\bibitem{SheTun95}
H.~D. Sherali and C.~H. Tuncbilek.
\newblock A reformulation-convexification approach for solving nonconvex
  quadratic programming problems.
\newblock {\em Journal of Global Optimization}, 7(1):1--31, 1995.

\bibitem{SheAda90}
H.D. Sherali and W.P. Adams.
\newblock A hierarchy of relaxations between the continuous and convex hull
  representations for zero-one programming problems.
\newblock {\em SIAM Journal of Discrete Mathematics}, 3(3):411--430, 1990.

\bibitem{tardella1988class}
F.~Tardella.
\newblock On a class of functions attaining their maximum at the vertices of a
  polyhedron.
\newblock {\em Discrete applied mathematics}, 22(2):191--195, 1988.

\bibitem{Hertog21}
J.~Zhen, D.~de~Moor, and D.~den Hertog.
\newblock An extension of the reformulation-linearization technique to
  nonlinear optimization.
\newblock {\em Available at Optimization Online}, 2021.

\end{thebibliography}

\end{document}